\documentclass[12pt,reqno]{amsart}
\usepackage{amsmath,amsopn,amssymb,amsthm,multicol}
\usepackage{hyperref}
\usepackage[all]{xy}
\usepackage{graphicx}
\usepackage{graphics}
\usepackage{xcolor}

\DeclareMathOperator{\Ad}{Ad}
\DeclareMathOperator{\ad}{ad}

\DeclareMathOperator{\tr}{tr}

\DeclareMathOperator{\Ric}{Ric}

\newcommand{\fr}{\mathfrak}
\newcommand{\al}{\alpha}
\newcommand{\be}{\beta}

\newcommand{\bb}{\mathbb}

\DeclareMathOperator{\SO}{SO}
\DeclareMathOperator{\s}{S}

\DeclareMathOperator{\SU}{SU}
\DeclareMathOperator{\U}{U}

 \newtheorem{lemma} {Lemma} [section]
\newtheorem{theorem}[lemma]{Theorem} 
 
\newtheorem{prop} [lemma]{Proposition}

\usepackage{array}
\makeatletter
\newcommand{\thickhline}{%
    \noalign {\ifnum 0=`}\fi \hrule height 1pt
    \futurelet \reserved@a \@xhline
}
\newcolumntype{"}{@{\hskip\tabcolsep\vrule width 1pt\hskip\tabcolsep}}
\makeatother

\begin{document}
\title{Non naturally reductive  Einstein metrics on $\SU(N)$ via generalized flag manifolds}
\author{Andreas Arvanitoyeorgos (corresponding author), Yusuke Sakane and Marina Statha}
\address{University of Patras, Department of Mathematics, GR-26500 Rion, Greece AND\newline Hellenic Open University, Aristotelous 18, GR-26335 Patras, Greece}
\email{arvanito@math.upatras.gr}
 \address{Osaka University, Department of Pure and Applied Mathematics, Graduate School of Information Science and Technology, Suita, 
Osaka 565-0871, Japan}
 \email{sakane@math.sci.osaka-u.ac.jp}
\address{University of Thessaly, Department of Mathematics, GR-35100 Lamia, Greece} 
\email{marinastatha@uth.gr} 
\medskip

\begin{abstract}
  
We obtain new invariant Einstein metrics  on the compact Lie group
 $\SU(N)$  which are not naturally reductive.
 This is achieved by  using the generalized flag manifold $G/K=\SU(k_1+\cdots +k_p)/\s(\U(k_1)\times\cdots\times\U(k_p))$ and by   taking an appropriate choice of  orthogonal basis of the center of Lie subalgebra $\frak k$ for $K$, which poses
  certain symmetry conditions to the $\Ad(K)$-invariant metrics of  $\SU(N)$. We also study the isometry problem for the Einstein metrics found.


   \end{abstract}

\maketitle


\section{Introduction}
\markboth{Arvanitoyeorgos, Sakane and Statha}{New homogeneous Einstein metrics on $\SU(N)$}

A Riemannian manifold $(M, g)$ is called Einstein if it has constant Ricci curvature, i.e. $\Ric_{g}=\lambda\cdot g$ for some $\lambda\in\bb{R}$ (cf. \cite{Be}). 
General existence results are difficult to obtain, so one may impose some symmetry conditions to try to prove existence of Einstein metrics.
For the case of a Riemannian homogeneous space $M=G/K$, where $G$ is a Lie group and $H$ a closed subgroup of $G$, the aim is the prove existence of $G$-invariant 
Einstein metrics and, if possible, classify such metrics. We refer to  \cite{W1}, \cite{W2} and \cite{A} for various results about homogeneous Einstein metrics
 
 For  a 
compact Lie group the problem  focuses  to finding  all left-invariant Einstein metrics.
This is  more subtle, since the space of all left-invariant metrics on a Lie group up to isometry and scaling is quite difficult to study. 
It is well known that a compact and semisimple Lie group equipped with a bi-invariant metric is Einstein.
 In \cite{DZ}  J.E. D'Atri and W. Ziller found a large number of naturally reductive metrics on compact Lie groups (classical and exceptional) and
  they raised the question of existence of left-invariant Einstein metrics which are not naturally reductive.
  The aim of the present paper is to prove
existence of new non naturally reductive left-invariant Einstein metrics on the special unitary group $\SU(N)$, and extend some previous results
  and techniques. 

A method to  look for left-invariant Einstein metrics on a compact Lie group $G$ is to restrict the search into a  subset of corresponding scalar products on the Lie algebra 
$\fr{g}$ of $G$ as follows. 
  We consider a homogeneous space $G/K$ and the natural  submersion $G\to G/K$ with fiber $K$.
  We decompose the Lie subalgebra of $K$ as $\fr{k}=\fr{k}_0\oplus\fr{k}_1 \oplus\cdots \oplus\fr{k}_p$, where $\fr{k}_0$ is the Lie algebra of the center of $K$ and
  $\fr{k}_i$ are simple ideals.   We also assume that the  tangent space $\fr{m}\cong T_0(G/K)$ of $G/K$ decomposes into irreducible and non equivalent $\Ad(K)$-modules
  $\fr{m}=\fr{m}_1\oplus\cdots\oplus\fr{m}_s$.
  
    Then the tangent space of $G$ is decomposed 
  into a direct sum of irreducible $\Ad(K)$-modules 
  $\fr{g}=\fr{k}_0\oplus\fr{k}_1 \oplus\cdots \oplus\fr{k}_p\oplus\fr{m}_1\oplus\cdots\oplus\fr{m}_s$ 
  and we  restrict ourselves to \color{black}
   left-invariant metrics on $G$ which are determined by $\Ad(K)$-invariant scalar products on $\fr{g}$. 
Then calculations for the Ricci curvature become simpler, by taking into account  well known formula  for the Ricci curvature  for $G$-invariant metrics on a homogeneous space  
 whose isotropy representation decomposes into \color{black} a sum of non equivalent irreducible summands
(cf. \cite{PS}).

 Concerning previous results about non naturally reductive left-invariant Einstein metrics on the special unitary group, the first existence result   was obtained by K. Mori in 
  \cite{M}, who  proved existence of such metrics $\SU(4+n)$ ($n\ge 2$).  
  There, the Lie group  $\SU(4+n)$ was considered as a total space over  generalized flag manifolds  $G/K=\SU(4+n)/\s(\U(2)\times\U(2)\times \U(n))$ with three isotropy summands.
In \cite{ASS1} the authors obtained invariant Einstein metrics on the compact Lie group $\SU(2n)$
$(n>3)$ which are not naturally reductive, by considering
$\SU(2n)$ as total space over the generalized Wallach spaces $G/K=\SU(2n)/\U(n)$. 
Note that in that case the center of the Lie algebra of $K=\U(n)$ is one-dimensional, which makes description of Ricci tensor of invariant metrics accessible.

 In \cite{ASS2} the authors  studied  non naturally reductive left-invariant Einstein metrics on the special unitary groups $G = \SU(\ell +m+ n)$ by 
 using the
generalized flag manifolds $G/K = \SU(\ell + m + n)/\s(\U(\ell) \times\U(m)\times \U(n))$.
In this case  the center of the Lie algebra of $K=\s(\U(\ell) \times\U(m)\times \U(n))$ is two-dimensional which makes description of  Ricci tensor more complicated.
We proved that $\SU(4+n)$ $(n\ge 2)$ and also $\SU(5)$ admit a certain number of non naturally reductive left-invariant Einstein metrics.
Notice that for the Lie groups $\SU(3)$, $\SU(4)$ and $\SU(2)\times\SU(2)$, the number of left-invariant Einstein metrics  is still unknown.

For  non naturally reductive left-invariant Einstein metrics on other classical and exceptional Lie groups we refer to \cite{AMS},  \cite{ASS2}, \cite{ASS1}, \cite{ASS4}, \cite{CL}, \cite{CCD}.



   In the present paper we  generalize the above approaches and prove existence of new non naturally reductive left-invariant Einstein metrics on the special unitary group $\SU(N)$, $N=k_1+k_2+\cdots +k_p$  (where $k_i \geq 2$ ($i =1, \dots, p$)),  by using the generalized flag manifold 
 $G/K=\SU(k_1+\cdots +k_p)/\s(\U(k_1)\times\cdots\times\U(k_p))$.   
 
 Note that 
 the isotropy representation $\frak{m} =\sum_{1\leq r<s \leq p} \frak{m}_{r s}$  of $G/K$ does not contain equivalent irreducible summands as $\Ad(K)$-modules. 
 The Lie algebra  $\frak{su}(N)$ is decomposed into 
$$\frak{su}(N) =\frak{k}_0+ \sum_{1\leq i \leq p}\frak{su}(k_i)+ \sum_{1\leq r<s \leq p} \frak{m}_{r s}$$ as $\Ad(K)$-module, 
 where $\frak{k}_0$ is the center $\frak{k}$ ($\dim \frak{k}_0 =  p-1$).  Thus we see that left-invariant metrics $g$ are determined by  $\Ad(K)$-invariant inner products of the form
 \begin{equation*}
  \langle \  \ ,\ \ \rangle =  < \ \, \ \ >|_{\frak{k}_0}+ \sum_{1\leq i \leq p} x_i B|_{\frak{su}(k_i)}+ \sum_{1\leq r<s \leq p}x_{r s} B|_{\frak{m}_{r s}}
 \end{equation*}
 where $< \ \, \ \ >|_{\frak{k}_0}$ denotes an arbitrary inner product on the center $\frak{k}_0$.  


 We take $k_2=\cdots =k_p=k \ge 2$.
Then the main result  is the following:

\begin{theorem}\label{main}  Let $p\ge 3$. 

\noindent
 {\rm 1)} If $k_1 > k \ge 2$, then  the Lie group  $\SU(N)$ {\em ($N=k_1 +(p-1)k$)} admits at least two  $\Ad(\s(\U(k_1)\times \underbrace{ \U(k)\times \cdots \times \U(k)}_{p-1})$-invariant Einstein metrics, which are not naturally reductive. 

\noindent
{\rm 2)} If $k_1=k$, then the Lie group  $\SU(N)$ {\em ($N=p \,k$)} admits at least one $\Ad( \s( \underbrace{ \U(k)\times \cdots \times \U(k)}_{p})$-invariant Einstein metric, which is not naturally reductive. 
\end{theorem}

\smallskip

\noindent
  For $p \geq 4$ these metrics are different from the ones obtained in \cite{ASS1} and \cite{ASS2}. 

\smallskip
 In addition, it is possible to show that, for $k_1 \geq 8 k p$, $k \geq 2$ and $ p \geq 3$,   the Lie group  $\SU(N)$  ($N=k_1 +(p-1)k$) admits at least four  $\Ad(\s(\U(k_1)\times \underbrace{ \U(k)\times \cdots \times \U(k)}_{p-1})$-invariant Einstein metrics, which are not naturally reductive. 
 We only give a  sketch of its proof in  Remark 1 of Section 4.

 At this point we would like to comment on a seemingly similarity of Theorem 1.1 to Theorem 1.1 in our recent work \cite{ASS4}, where we studied existence of left-invariant
 Einstein metrics on $\SO(n)$, which are not naturally reductive.
 There, we used a real flag manifold $G/K$ in the decomposition 
 $\fr{g}=\fr{k}_1 \oplus\cdots \oplus\fr{k}_p\oplus\fr{m}_1\oplus\cdots\oplus\fr{m}_s$ of the Lie algebra of $\SO(n)$.   
 The  Lie algebra of $K$ for  a real flag manifold $G/K$ has no center, whereas in the present corresponding decomposition for $\SU(N)$,  we use a complex flag manifold
 $G/K$.  Here, the dimension of the center of the Lie algebra of $K$ is   $p -1$ ($ p \geq 3$).  This makes the description of Ricci tensor for left-invariant metrics on $G$ corresponding to $\Ad(K)$-invariant scalar products on $\fr{g}$ more complicated for $\SU(N)$, that is, Ricci tensor may have non diagonal part. 
 This   requires some non trivial study  of the action of the center of  $K$ (cf. Lemma 3.4, Lemma 3.6 and  Proposition 3.8).
 
 Another point that one may observe, is that in  part 1) of both  theorems we prove existence of at least two invariant Einstein metrics.
 This is because in both cases the Einstein equation reduces to a polynomial parametric system of equations, and the two Einstein metrics are obtained as positive solutions of a polynomial equations of degree 8 in case for  $\SO(n)$, but a polynomial equations of degree 16 in case for  $\SU(n)$.  By making appropriate choices of the parameters of the system, we are able to prove existence of more solutions.

\smallskip
\noindent
The paper is organized as follows:  	In Section 2 we recall a general expression for the Ricci tensor for  diagonal invariant metrics  on homogeneous spaces.  In Section 3 we consider a special decomposition of the Lie algebra $\fr{g}=\fr{su}(N)$ of the special unitary group $G=\SU(N)=\SU(k_1+\cdots +k_p)$ by using the flag manifold 
$G/K=\SU(N)/\s(\U(k_1)\times\cdots\times\U(k_p))$ (cf. (\ref{decom_su(n)})).  Then we make an appropriate choice of  orthogonal basis of the center of the Lie subalgebra $\frak k$ for $K$ given in (\ref{basis}). 
Associated to the decomposition (\ref{decom_su(n)}) of $\fr{g}=\fr{su}(N)$, we define a class of left-invariant metrics on $\SU(N)$, determined by $\Ad(K)$-invariant scalar products
on $\fr{g}$, given by expression (\ref{metric_g}).  These metrics are diagonal. 

We show that,  to obtain an Einstein metric on $\SU(N)$ that is,  
off-diagonal part of Ricci tensor to be zero,
  certain symmetry conditions are needed for the $\Ad(K)$-invariant metrics on  $\SU(N)$ (cf. Lemma \ref{consditions_r(H_i, H_j)=0}).
  
 Next, we take $k_1=k_2=\cdots =k_p=k\ge 2$ and consider a special case of the metric (\ref{metric_g})
on the Lie group $G=\SU(k_1+(p-1)k)$ (cf. (\ref{metric_g111})).  This metric depends on positive
variables $y_1, y_2, x_1, x_2, x_{12}, x_{23}$.
 Then the Ricci components $rr_1, rr_2, r_1, r_2, r_{12}, r_{23}$ of this metric are computed in Proposition \ref{ric_component}. 

In Section 4 we find left-invariant Einstein metrics on $\SU(k_1+(p-1)k)$ by making an extensive analysis of the parametric system of equations
 $\{r_{12} - r_{23} = 0,\  r_{23} - r_2 = 0, \  r_1 - r_2 = 0, \
 r_2 - rr_1 = 0,  rr_1 - rr_2= 0\}$.
 The approach is the following: We normalize the equations by taking $x_{23}=1$ and consider two cases.
 In Case 1 we assume that $x_2=1$ and we obtain either bi-invariant Einstein metrics or precisely two left-invariant Einstein metrics.
 In Case 2 we prove existence of at least two left-invariant Einstein metrics.

  In Section 5 we provide a necessary and sufficient   condition so that these left-invariant metrics (\ref{metric_g111})  are naturally reductive (cf. Proposition \ref{reductive}). Then we use this, to determine which of the Einstein metrics on  $G=\SU(N)=\SU(k_1 +(p-1)k)$ found in Section 4, are not naturally reductive.
  Finally, in Section 6 we consider  the isometry problem for these  Einstein metrics by comparing Einstein constants.

\section{The Ricci tensor for reductive homogeneous spaces}
 We recall an expression for the Ricci tensor for a $G$-invariant Riemannian
metric on a reductive homogeneous space whose isotropy representation
is decomposed into a sum of non equivalent irreducible summands.

Let $G$ be a compact semisimple Lie group, $K$ a connected closed subgroup of $G$  and  
let  $\frak g$ and $\frak k$  be  the corresponding Lie algebras. 
The Killing form $B$ of $\frak g$ is negative definite, so we can define an $\mbox{Ad}(G)$-invariant inner product $-B$ on 
  $\frak g$. 
Let $\frak g$ = $\frak k \oplus
\frak m$ be a reductive decomposition of $\frak g$ with respect to $-B$ so that $\big[\,\frak k,\, \frak m\,\big] \subset \frak m$ and
$\frak m\cong T_o(G/K)$.  
 We decompose  $ {\frak m} $ into irreducible non equivalent $\mbox{Ad}(K)$-modules as follows: \ 
\begin{equation}\label{iso}
{\frak m} = {\frak m}_1 \oplus \cdots \oplus {\frak m}_q.
\end{equation} 
Then for the decomposition (\ref{iso}) any $G$-invariant metric on $G/K$ can be expressed as  
\begin{eqnarray}
 \langle  \,\,\, , \,\,\, \rangle  =  
x_1   (-B)|_{\mbox{\footnotesize$ \frak m$}_1} + \cdots + 
 x_q   (-B)|_{\mbox{\footnotesize$ \frak m$}_q},  \label{eq2}
\end{eqnarray}
for positive real numbers $(x_1, \dots, x_q)\in \mathbb{R}^{q}_{+}$. 
If, for the decomposition (\ref{iso}) of $\frak m$,  the Ricci tensor $r$ of a $G$-invariant Riemannian metric $ \langle  \,\,\, , \,\,\, \rangle$ on $G/K$  satisfies $r({\frak m}_i, {\frak m}_j) = (0)$ for $ i \neq j$, then the Ricci tensor $r$ is of the same form as (\ref{eq2}), that is 
 \[
 r=z_1 (-B)|_{\mbox{\footnotesize$ \frak m$}_1}  + \cdots + z_{q} (-B)|_{\mbox{\footnotesize$ \frak m$}_q} ,
 \]
 for some real numbers $z_1, \ldots, z_q$.

Let $\lbrace e_{\alpha} \rbrace$ be a $(-B)$-orthonormal basis 
adapted to the decomposition of $\frak m$,    i.e. 
$e_{\alpha} \in {\frak m}_i$ for some $i$, and
$\alpha < \beta$ if $i<j$. 
We put ${A^\gamma_{\alpha
\beta}}= -B \big(\big[e_{\alpha},e_{\beta}\big],e_{\gamma}\big)$ so that
$\big[e_{\alpha},e_{\beta}\big]
= \displaystyle{\sum_{\gamma}
A^\gamma_{\alpha \beta} e_{\gamma}}$ and set 
$\displaystyle{k \brack {ij}}=\sum (A^\gamma_{\alpha \beta})^2$, where the sum is
taken over all indices $\alpha, \beta, \gamma$ with $e_\alpha \in
{\frak m}_i,\ e_\beta \in {\frak m}_j,\ e_\gamma \in {\frak m}_k$ (cf.\,\cite{WZ}).  
Then the positive numbers $\displaystyle{k \brack {ij}}$ are independent of the 
$(-B)$-orthonormal bases chosen for ${\frak m}_i, {\frak m}_j, {\frak m}_k$,
and 
$\displaystyle{k \brack {ij}}\ =\ \displaystyle{k \brack {ji}}\ =\ \displaystyle{j \brack {ki}}.  
 \label{eq3}
$

Let $ d_k= \dim{\frak m}_{k}$. Then we have the following:

\begin{lemma}\label{ric2}{\em(\cite{PS})}
The components ${ r}_{1}, \dots, {r}_{q}$ 
of the Ricci tensor ${r}$ of the metric $ \langle  \,\,\, , \,\,\, \rangle $ of the
form {\em (\ref{eq2})} on $G/K$ are given by 
\begin{equation}
{r}_k = \frac{1}{2x_k}+\frac{1}{4d_k}\sum_{j,i}
\frac{x_k}{x_j x_i} {k \brack {ji}}
-\frac{1}{2d_k}\sum_{j,i}\frac{x_j}{x_k x_i} {j \brack {ki}}
 \quad (k= 1, \dots, q),    \label{eq51}
\end{equation}
where the sum is taken over $i, j =1, \dots , q$.
\end{lemma} 
If, for the decomposition (\ref{iso}) of $\frak m$,  the Ricci tensor $r$ of a $G$-invariant Riemannian metric $ \langle  \,\,\, , \,\,\, \rangle$ on $G/K$  satisfies $r({\frak m}_i, {\frak m}_j) =0$ for $ i \neq j$, 
then, by Lemma \ref{ric2}, it follows that $G$-invariant Einstein metrics on $M=G/K$ are exactly the positive real solutions $(x_1, \ldots, x_q)\in\mathbb{R}^{q}_{+}$  of the system of equations $\{r_1=\lambda, \, r_2=\lambda, \, \ldots, \, r_{q}=\lambda\}$, where $\lambda\in \mathbb{R}_{+}$ is the Einstein constant.

\section{A class of left-invariant metrics on $\SU(N)=\SU(k_1+k_2+\dots +k_p)$  and Ricci tensor}

\subsection{A decomposition of the tangent space}

We will describe a decomposition of the tangent space of the Lie group $\SU (N)$.
 We consider the closed subgroup $K = \s(\U(k_1)\times\U(k_2)\times\cdots\times\U(k_p))$ of $G =\SU(N) =  \SU(k_1 + \cdots+ k_p)$
 with $k_1, k_2, \dots,  k_p \ge 1$,
 where the embedding of $K$ in $G$ is diagonal.
The homogeneous space $G/K$ obtained is known as generalized flag manifold.  
  Then the tangent space $\fr{su}(k_1 + \cdots+ k_p)$ of the special unitary group $G = \SU(k_1 + \cdots + k_p)$ can be written as a direct sum of  $\Ad(K)$-invariant modules as 
  \begin{equation}\label{dia}
\fr{su}(k_1 + \cdots + k_p) = \fr{k}\oplus\fr{m}, 
\end{equation}
where   $\fr{m}$ corresponds to the tangent space of $G/K$. 

   Then the tangent space $ \fr{m}$ of $G/K$ 
   is given by  $\fr{k}^{\perp} $ in $ \fr{g} = \fr{su}(k_1+ \cdots+k_p)$ with respect to the $\mbox{Ad}(G)$-invariant inner product $-B$. If we denote by $M(q, r)$ the set of all $q \times r$ complex matrices, then we see that 
  $ \fr{m}$ is given by

\begin{equation*}
 \fr{m}=\left\{\left( 
 \begin{array}{cccccc}
  0 & A_{1\,2} &  A_{1\,3}  & \cdots & A_{1\, p-1}& A_{1\,p} \\ 
  -{\bar A_{1\,2}}^t & 0 &  A_{2\,3}  & \cdots & A_{2\, p-1}& A_{2\, p}\\
   -{\bar A_{1\,3}}^t & -{\bar A_{2\,3}}^t &  0  & \cdots & A_{3\, p-1}& A_{3 \,p}\\
    \vdots & \vdots &  \vdots  & \ddots & \vdots & \vdots\\
    \vspace{2pt}
      -{\bar A_{1\, p-1}}^t & -{\bar A_{2\, p-1}}^t & -{\bar A_{3 \,p-1}}^t & \cdots  & 0 & A_{ p-1\, p}
\\ 
   -{\bar A_{1\,p}}^t & -{\bar A_{2\, p}}^t &   -{\bar A_{3\, p}}^t  &\cdots  & -{\bar A_{p-1\,  p}}^t & 0
  \end{array}
  \right):
  \begin{array}{c} A_{r  s}\in M(k_r, k_s) \\
  1 \leq r < s \leq p
  \end{array}
  \right\}.
 \end{equation*}

We set, for $r < s$, 
\begin{equation}\label{m(rs)}
 \fr{m}_{rs}=\left\{\left( 
 \begin{array}{cccc}
  0 & \cdots &    & 0 \\ 
 \vdots & \ddots & A_{rs} & \vdots\\
   & -{\bar A_{rs}}^t &  & \\
  0 & \cdots &  & 0 
  \end{array}
  \right): A_{rs}\in M(k_r, k_s)
  \right\}.
 \end{equation} 
 We take a Weyl basis  $\{ X_{\al \be}^{(r s)}, Y_{\al \be}^{(r s)} ; 1 \leq \al \leq k_r,  1 \leq \be \leq k_s \}$ for $ \fr{m}_{rs}$.  That is, let $E_{\al \be}^{(r s)} \in M(k_r, k_s)$ be the matrix with  1 in the $\al$-th row and $\be$-th column. We denote by $X_{\al \be}^{(r s)}$ the element in $ \fr{m}_{rs}$ with  $A_{rs} = E_{\al \be}^{(r s)}$ in (\ref{m(rs)}),  and by $Y_{\al \be}^{(r s)}$ the element in $ \fr{m}_{rs}$ with  $A_{rs} = \sqrt{-1} E_{\al \be}^{(r s)}$  in (\ref{m(rs)}). 
 

 Note that $\fr{k}= \fr{s}(\fr{u}(k_1)\oplus\cdots\oplus\fr{u}(k_p))= \fr{c}\oplus\fr{su}(k_1)\oplus\cdots\oplus\fr{su}(k_p)$ where $\fr{c}$ is the $(p-1)$-dimensional center of $\fr{k}$.   
 For $i=1, \dots , p$ we also set $\fr{m}_i = \fr{su}(k_i)$. 
 Then the decomposition (\ref{dia}) of the tangent space of the special unitary group $G= \SU(k_1+\cdots+k_p)$ takes the form
\begin{equation}\label{decom_su(n)}
 \fr{su}(k_1+\cdots +k_p) =\fr{c}\oplus \fr{m}_1 \oplus \fr{m}_2 \oplus \cdots \oplus\fr{m}_p \oplus  \bigoplus_{1\le r< s\le p}\fr{m}_{r s}. 
\end{equation}

Now we consider an orthogonal basis $\{H_1, H_2, \dots, H_{p-1} \}$ of the center $\fr{c}$ with respect to inner product $-B$,   
which will be convenient for our study.  We set $K_j = k_1+k_2+\cdots+ k_j$ for $j =1, \ldots, p-1$. 

Let 
 \begin{equation*}
 H_{1}=\left( 
 \begin{array}{cc}
\displaystyle\frac{\sqrt{-1}}{k_1} \text{\large{I}}_{k_1} &  \text{{\large{0}}} \\ 
  \text{{\large{0}}} & \displaystyle - \frac{\sqrt{-1}}{N-K_1} \text{\large{I}}_{N- K_1}  
  \end{array}
  \right),  \quad 
 H_{2}=\left( 
 \begin{array}{ccc}
 \displaystyle  \text{{\large{0}}}_{k_1}&   \text{{\large{0}}} &  \text{{\large{0}}} \\ 
\displaystyle   \text{{\large{0}}}& \displaystyle   \frac{\sqrt{-1}}{k_2} \text{\large{I}}_{k_2} &  \text{{\large{0}}} \\ 
  \text{{\large{0}}} &  \displaystyle   \text{{\large{0}}} &\displaystyle - \frac{\sqrt{-1}}{N-K_2} \text{\large{I}}_{N- K_2}  
  \end{array}
  \right),
 \end{equation*} 
  \begin{equation}\label{basis}
  \dots, \quad 
   H_{j}=\left( 
 \begin{array}{ccc}
 \displaystyle  \text{{\large{0}}}_{K_{j-1}}&   \text{{\large{0}}} &  \text{{\large{0}}} \\ 
\displaystyle   \text{{\large{0}}}& \displaystyle   \frac{\sqrt{-1}}{k_j} \text{\large{I}}_{k_j} &  \text{{\large{0}}} \\ 
  \text{{\large{0}}} &  \displaystyle   \text{{\large{0}}} &\displaystyle - \frac{\sqrt{-1}}{N-K_j} \text{\large{I}}_{N- K_j}  
  \end{array}
  \right), \quad \dots\quad , 
 \end{equation} 
  \begin{equation*}
 H_{p-2}=\left( 
 \begin{array}{ccc}
 \displaystyle  \text{{\large{0}}}&   \text{{\large{0}}} &  \text{{\large{0}}} \\ 
\displaystyle   \text{{\large{0}}}& \displaystyle   \frac{\sqrt{-1}}{k_{p-2}} \text{\large{I}}_{k_{p-2}} &  \text{{\large{0}}} \\ 
  \text{{\large{0}}} &  \displaystyle   \text{{\large{0}}} &\displaystyle - \frac{\sqrt{-1}}{N-K_{p-2}} \text{\large{I}}_{N- K_{p-2}}  
  \end{array}
  \right),  
\end{equation*} 
 \begin{equation*} 
  H_{p-1}=
  \left( 
 \begin{array}{ccc}
 \displaystyle  \text{{\large{0}}}_{K_{p-2}}&   \text{{\large{0}}} &  \text{{\large{0}}} \\ 
\displaystyle   \text{{\large{0}}}& \displaystyle   \frac{\sqrt{-1}}{k_{p-1}} \text{\large{I}}_{k_{p-1}} &  \text{{\large{0}}} \\ 
  \text{{\large{0}}} &  \displaystyle   \text{{\large{0}}} &\displaystyle - \frac{\sqrt{-1}}{N-K_{p-1}} \text{\large{I}}_{N- K_{p-1}}  
  \end{array}
  \right)
=  \left( 
 \begin{array}{ccc}
 \displaystyle  \text{{\large{0}}}_{K_{p-2}}&   \text{{\large{0}}} &  \text{{\large{0}}} \\ 
\displaystyle   \text{{\large{0}}}& \displaystyle   \frac{\sqrt{-1}}{k_{p-1}} \text{\large{I}}_{k_{p-1}} &  \text{{\large{0}}} \\ 
  \text{{\large{0}}} &  \displaystyle   \text{{\large{0}}} &\displaystyle - \frac{\sqrt{-1}}{k_{p}} \text{\large{I}}_{k_p}  
  \end{array}
  \right), 
 \end{equation*} 
  where $\text{\large{I}}_k$ and $\text{{\large{0}}}_k$ denote the  identity and zero matrices of size $k$ respectively.

Using the orthogonal basis $\{H_1, H_2, \dots, H_{p-1} \}$ of  the center $\fr{c}$, the action $\ad(\fr{c})$  on $ \fr{m} =\bigoplus_{1\le r< s \le p}\fr{m}_{r s}$ is given as  follows: 
%
%
%
\begin{lemma}\label{action_center} 
Let $\{X_{\al \be}^{(r s)},  Y_{\al \be}^{(r s)} : 1\leq \al \leq k_r, 1 \leq \be \leq k_s \}$ be the Weyl basis of 
$ \fr{m}_{r s}$.  For a fixed $j$ $(j = 1, \dots, p-1)$, we have: 
\begin{center}
\begin{tabular}{ll} \vspace{5pt}
 $( 1 )$ &\  For $1\leq r  \leq j -1$,  \\ \vspace{5pt}
& $[ H_j, \fr{m}_{r s}] =( 0 )$      for $  s \leq j -1$,   \\ \vspace{5pt}
& $  [ H_j, X_{\al \be}^{(r j)}] =  -\frac{1}{k_j}Y_{\al \be}^{(r j)}$, 
 $  [ H_j, Y_{\al \be}^{(r j)}] =  \frac{1}{k_j}X_{\al \be}^{(r j)}$, 
\\  \vspace{5pt}
& $  [ H_j, X_{\al \be}^{(r s)}] = \frac{1}{N - K_j} Y_{\al \be}^{(r s)}$, 
 $[ H_j, Y_{\al \be}^{(r s)}] = -\frac{1}{N - K_j} X_{\al \be}^{(r s)}$     for  
 $j+1\leq  s \leq  p$.  
 \\ \vspace{5pt}
$( 2 )$ &  For $ r =j$,  \\ \vspace{5pt}
& $ [ H_j, X_{\al \be}^{(j s)}] = \left( \frac{1}{k_j}+\frac{1}{N - K_j}\right)Y_{\al \be}^{(j s)}$,  
 $[ H_j, Y_{\al \be}^{(j s)}] =  -\left( \frac{1}{k_j}+\frac{1}{N - K_j}\right)X_{\al \be}^{(j s)}$     
  for $j+1 < s \leq p$,  
  \\ \vspace{5pt}
 $( 3 )$  &  For $j+1\leq  r < s $, 
\\ \vspace{5pt}
& $[ H_j, \fr{m}_{r s}] =( 0 )$.         
\end{tabular}
\end{center}  
\end{lemma}
\begin{proof}  By a straightforward computation, we see our claim. 
\end{proof}

We  also see that the following relations hold:
\begin{lemma}\label{brackets}  
The submodules  $\fr{c}$, $\fr{m}_i$ and $\fr{m}_{r s}$ in the decomposition {\em (\ref{decom_su(n)})}  satisfy the following bracket relations:
\begin{center}
\begin{tabular}{lll}
$[ \fr{m}_i, \fr{m}_i]  \subset \fr{m}_i,$  &   
$[ \fr{m}_i, \fr{m}_{i s}]  \subset \fr{m}_{i s},$  
 & $[ \fr{m}_i, \fr{m}_{r i}]  \subset  \fr{m}_{r i},$ \\
$[ \fr{m}_{i j}, \fr{m}_{i s}]  \subset \fr{m}_{j s}$  for $j < s$,   &  $[ \fr{m}_{i j}, \fr{m}_{i s}]  \subset \fr{m}_{s j}$  for $s < j$, & $[ \fr{m}_{i j}, \fr{m}_{r i}]  \subset \fr{m}_{r j}$,    \\ 
$[\fr{m}_{i j}, \fr{m}_{j s}]  \subset   \fr{m}_{i s}$,  & $[\fr{m}_{i j}, \fr{m}_{r j}]  \subset   \fr{m}_{i r}$ for $i < r$, & $[\fr{m}_{i j}, \fr{m}_{r j}]  \subset   \fr{m}_{r i}$ for $r < i$, \\ 
$[ \fr{m}_{r s}, \fr{m}_{r s}] \subset  \fr{c} + \fr{m}_{r}+\fr{m}_{s}$,  &  $[\fr{c}, \fr{m}_{r s}] \subset \fr{m}_{r s}$ &  
\end{tabular}
\end{center} 
and the other bracket relations are zero.
\end{lemma}
\begin{proof}  By a straightforward computation, we see our claim. 
\end{proof}

Note that the irreducible $\Ad(K)$-submodules $\fr{m}_{r s}$ are mutually non equivalent, so any
 $G$-invariant metric on $G/K$ is determined by an $\Ad(K)$-invariant scalar product
$
\sum_{  r<s }x_{r s}  B |_{ \fr{m}_{r s}}
$, where $x_{r s}$ are positive real numbers. 

We denote by $\fr{c}_j$ the 1-dimensional space ${\bb R} H_j$  for $j=1, \dots, p-1$. Then we have $\fr{c} = \sum_{j=1}^{p-1} \fr{c}_j$.

\subsection{Invariant metrics and Ricci tensor}
We  define  $\Ad(K)$-invariant scalar products $g$ on $ \fr{g} = \fr{su}(k_1+ \cdots+k_p)$ by
 \begin{equation}\label{metric_g} 
 g =  \sum_{j=1}^{p-1}y_j  B |_{ \fr{c}_j }+  \sum_{\ell=1}^{p} x_\ell B |_{\fr{m}_\ell}+ \sum_{1\le r<s \le p}x_{r s} B |_{ \fr{m}_{r s}}. 
 \end{equation}
 
 We compute the components of the Ricci tensor $r$ of the metric (\ref{metric_g}). 
 
\begin{lemma}\label{r(H_i, H_j)} 
Let $\{{\widetilde{X}_{\al \be}^{(r s)}},  \widetilde{Y}_{\al \be}^{(r s)} : 1\leq \al \leq k_r, 1 \leq \be \leq k_s \}$ be the orthonormal basis of $\fr{m}_{r s}$ with respect to the inner product $B$ defined from the Weyl basis  $\{ X_{\al \be}^{(r s)}, Y_{\al \be}^{(r s)} ; 1 \leq \al \leq k_r,  1 \leq \be \leq k_s \}$. 
For the elements $H_i, H_j$ in $\fr{c}$, we have 
\begin{eqnarray}
 r(H_i, H_j) = \frac{y_i  y_j}{4} \sum_{r < s}\frac{1}{x_{r s}^2} \Big( \sum_{\al, \be} \left(B( \widetilde{Y}_{\al \be}^{(r s)}, [H_i, \widetilde{X}_{\al \be}^{(r s)}]) B( \widetilde{Y}_{\al \be}^{(r s)}, [H_j, \widetilde{X}_{\al \be}^{(r s)}]) \right. \nonumber 
 \\ \left. 
 +  B( \widetilde{X}_{\al \be}^{(r s)}, [H_i, \widetilde{Y}_{\al \be}^{(r s)}]) B( \widetilde{X}_{\al \be}^{(r s)}, [H_j, \widetilde{Y}_{\al \be}^{(r s)}])\right)\Big). 
 \end{eqnarray}
 \end{lemma}
 \begin{proof}  
 It follows from an appropriate adjustment of the proof of Proposition 4.4 in \cite{ASS2} and we omit it.
\end{proof}

\begin{lemma}\label{consditions_r(H_i, H_j)=0} 
For the orthogonal basis $\{H_1, H_2, \dots, H_{p-1} \}$ of  the center $\fr{c}$, 
we have 

$ r(H_i, H_j) =0$  for $ 1 \leq i < j \leq p$,  if and only if 
\begin{eqnarray}\label{eqx_{ij}}
x_{i\,  i+1} = x_{i \, i+2} = \cdots = x_{i \,  p},  &  for & i = 1, 2, \dots,  p-1. 
 \end{eqnarray}
 \end{lemma}
 \begin{proof}  First we consider the case when $i=1$. We claim that $ r(H_1, H_j) =0$ for  $ 1 \leq j \leq p-1$ if and only if $x_{1\, 2} =   \cdots = x_{1\, p-1} =x_{1 \,  p} $. 
 
 From Lemma \ref{r(H_i, H_j)}, we see that $ r(H_1, H_j) =0$ if and only if 
\begin{eqnarray}\label{r(H_1, H_j)}  \sum_{r < s}\frac{1}{x_{r s}^2} \Big( \sum_{\al, \be} \left(B( \widetilde{Y}_{\al \be}^{(r s)}, [H_1, \widetilde{X}_{\al \be}^{(r s)}]) B( \widetilde{Y}_{\al \be}^{(r s)}, [H_j, \widetilde{X}_{\al \be}^{(r s)}]) \right. \nonumber 
 \\ \left. 
 +  B( \widetilde{X}_{\al \be}^{(r s)}, [H_1, \widetilde{Y}_{\al \be}^{(r s)}]) B( \widetilde{X}_{\al \be}^{(r s)}, [H_j, \widetilde{Y}_{\al \be}^{(r s)}])\right)\Big) =0. 
 \end{eqnarray}
Now we have $ [H_1, \widetilde{X}_{\al \be}^{(1 s)}] =  \left( \frac{1}{k_1}+\frac{1}{N - K_1}\right)\widetilde{Y}_{\al \be}^{(1 s)}$, $ [H_1, \widetilde{Y}_{\al \be}^{(1 s)}]$ $ =  -\left( \frac{1}{k_1}+\frac{1}{N - K_1}\right)\widetilde{X}_{\al \be}^{(1 s)}$ and,   for $r \geq 2$, $[ H_1, \fr{m}_{r s}] =( 0 )$  from Lemma \ref{action_center}. Thus  equation (\ref{r(H_1, H_j)}) becomes 
\begin{eqnarray}\label{r(H_1, H_j)2}  \sum_{1 < s}\frac{1}{x_{1 s}^2} \Big( \sum_{\al, \be} \left(B( \widetilde{Y}_{\al \be}^{(1 s)}, [H_1, \widetilde{X}_{\al \be}^{(1 s)}]) B( \widetilde{Y}_{\al \be}^{(1 s)}, [H_j, \widetilde{X}_{\al \be}^{(1 s)}]) \right. \nonumber 
 \\ \left. 
 +  B( \widetilde{X}_{\al \be}^{(1 s)}, [H_1, \widetilde{Y}_{\al \be}^{(1 s)}]) B( \widetilde{X}_{\al \be}^{(1 s)}, [H_j, \widetilde{Y}_{\al \be}^{(1 s)}])\right)\Big) =0. 
 \end{eqnarray} 
 By Lemma \ref{action_center} (1), we have   $[ H_j, \fr{m}_{1 s}] =( 0 )$  for $s \leq j -1$,
 $  [ H_j, X_{\al \be}^{(1 j)}] =  -\frac{1}{k_j}Y_{\al \be}^{(1 j)}$, 
 $  [ H_j, Y_{\al \be}^{(1 j)}] =  \frac{1}{k_j}X_{\al \be}^{(1 j)}$, and
 $  [ H_j, X_{\al \be}^{(1 s)}] = \frac{1}{N - K_j} Y_{\al \be}^{(1 s)}$, 
 $[ H_j, Y_{\al \be}^{(1 s)}] = -\frac{1}{N - K_j} X_{\al \be}^{(1 s)}$     for  
 $j+1\leq  s \leq  p$.   Hence, by (\ref{r(H_1, H_j)2}), we see that   $ r(H_1, H_j) =0$ if and only if 
 \begin{eqnarray*}  \left( \frac{1}{k_1}+\frac{1}{N - K_1}\right)\left(-\frac{1}{k_j}\frac{1}{x_{1 j}^2} 2 k_1 k_j
 + \frac{1}{N-K_j}\sum_{j+1 < s} \frac{1}{x_{1 s}^2} 2 k_1 k_s\right)=0. 
 \end{eqnarray*} 
 Therefore we obtain that  $ r(H_1, H_j) =0$ if and only if 
\begin{eqnarray}\label{r(H_1, H_j)3} 
- \frac{1}{x_{1 j}^2} 
 + \frac{1}{N-K_j}\sum_{j+1 < s} \frac{1}{x_{1 s}^2}  k_s =0. 
 \end{eqnarray}
 For $j=p-1$ we see that $0= \displaystyle - \frac{1}{x_{1 p-1}^2} 
 + \frac{1}{N-K_{p-1}} \frac{1}{x_{1 p}^2} k_p = - \frac{1}{x_{1 p-1}^2} 
 +  \frac{1}{x_{1 p}^2}$, and hence, $x_{1 p-1} = x_{1 p}$. Thus, from (\ref{r(H_1, H_j)3}), we see our claim  for $i=1$. 
 
 Now we consider the case when $i \geq 2$. We assume that the equations (\ref{eqx_{ij}})
hold for $i \leq \ell-1$ and  prove that (\ref{eqx_{ij}}) holds for $i =\ell$. 
 Indeed, we see that $ r(H_{\ell}, H_t) =0$ $(\ell  <  t)$ if and only if 
 \begin{eqnarray}\label{ell}  
 \sum_{r < s}\frac{1}{{x_{r s}}^2} \Big( \sum_{\al, \be} \left(B( \widetilde{Y}_{\al \be}^{(r s)}, [H_{\ell}, \widetilde{X}_{\al \be}^{(r s)}]) B( \widetilde{Y}_{\al \be}^{(r s)}, [H_t, \widetilde{X}_{\al \be}^{(r s)}]) \right. \nonumber 
 \\ \left. 
 +  B( \widetilde{X}_{\al \be}^{(r s)}, [H_{\ell}, \widetilde{Y}_{\al \be}^{(r s)}]) B( \widetilde{X}_{\al \be}^{(r s)}, [H_t, \widetilde{Y}_{\al \be}^{(r s)}])\right)\Big) =0. 
 \end{eqnarray}
  From Lemma \ref{action_center}  for $j=\ell$,  we see that equation (\ref{ell}) is divided into five sums: 
  \begin{eqnarray}\label{ell+1}  
  && 0 = \sum_{r \leq \ell-1, \, s \leq \ell-1}\frac{1}{{x_{r \,  s}}^2} \Big( \sum_{\al, \be} \left(B( \widetilde{Y}_{\al \be}^{(r \, s)}, [H_{\ell}, \widetilde{X}_{\al \be}^{(r \, s)}]) B( \widetilde{Y}_{\al \be}^{(r \, s)}, [H_t, \widetilde{X}_{\al \be}^{(r \, s)}]) \right. \nonumber 
 \\&& \left. 
 +  B( \widetilde{X}_{\al \be}^{(r\, s)}, [H_{\ell}, \widetilde{Y}_{\al \be}^{(r \, s}]) B( \widetilde{X}_{\al \be}^{(r \, s)}, [H_t, \widetilde{Y}_{\al \be}^{(r \, s)}])\right)\Big) \nonumber 
 \\
&& + \sum_{r \leq \ell}\frac{1}{{x_{r \,  \ell}}^2} \Big( \sum_{\al, \be} \left(B( \widetilde{Y}_{\al \be}^{(r \, \ell)}, [H_{\ell}, \widetilde{X}_{\al \be}^{(r \, \ell)}]) B( \widetilde{Y}_{\al \be}^{(r \, \ell)}, [H_t, \widetilde{X}_{\al \be}^{(r \, \ell)}]) \right. \nonumber 
 \\&& \left. 
 +  B( \widetilde{X}_{\al \be}^{(r\, \ell+1)}, [H_{\ell}, \widetilde{Y}_{\al \be}^{(r \, \ell )}]) B( \widetilde{X}_{\al \be}^{(r \, \ell)}, [H_t, \widetilde{Y}_{\al \be}^{(r \, \ell)}])\right)\Big) \nonumber 
 \\
 && + \sum_{r \leq \ell-1,\, \ell+1 \leq s}\frac{1}{{x_{r \, s}}^2} \Big( \sum_{\al, \be} \left(B( \widetilde{Y}_{\al \be}^{(r \, s)}, [H_{\ell}, \widetilde{X}_{\al \be}^{(r \, s )}]) B( \widetilde{Y}_{\al \be}^{(r \, s)}, [H_t, \widetilde{X}_{\al \be}^{(r \, s)}]) \right. \nonumber 
 \\
 && \left. 
 +  B( \widetilde{X}_{\al \be}^{(r\, s)}, [H_{\ell}, \widetilde{Y}_{\al \be}^{(r \, \ell )}]) B( \widetilde{X}_{\al \be}^{(r \, s)}, [H_t, \widetilde{Y}_{\al \be}^{(r \, s )}])\right)\Big) \nonumber 
 \\
&& + \sum_{r = \ell  < s}\frac{1}{{x_{ \ell  \, s}}^2} \Big( \sum_{\al, \be} \left(B( \widetilde{Y}_{\al \be}^{(\ell \, s)}, [H_{\ell}, \widetilde{X}_{\al \be}^{(\ell \, s)}]) B( \widetilde{Y}_{\al \be}^{(\ell \, s)}, [H_t, \widetilde{X}_{\al \be}^{(\ell \, s)}]) \right. \nonumber 
 \\&& \left. 
 +  B( \widetilde{X}_{\al \be}^{(\ell  \, s1)}, [H_{\ell }, \widetilde{Y}_{\al \be}^{(\ell  \, s)}]) B( \widetilde{X}_{\al \be}^{(\ell  \, s )}, [H_t, \widetilde{Y}_{\al \be}^{(\ell  \, s)}])\right)\Big) \nonumber 
\\ 
&& + \sum_{ \ell +1 \leq r < s}\frac{1}{{x_{r \,  s}}^2} \Big( \sum_{\al, \be} \left(B( \widetilde{Y}_{\al \be}^{(r \, s)}, [H_{\ell }, \widetilde{X}_{\al \be}^{(r \, s)}]) B( \widetilde{Y}_{\al \be}^{(r \, s)}, [H_t, \widetilde{X}_{\al \be}^{(r \, s)}]) \right. \nonumber 
 \\&& \left. 
 +  B( \widetilde{X}_{\al \be}^{(r\, s)}, [H_{\ell }, \widetilde{Y}_{\al \be}^{(r \, s)}]) B( \widetilde{X}_{\al \be}^{(r \, s)}, [H_t, \widetilde{Y}_{\al \be}^{(r \, s)}])\right)\Big). \nonumber 
 \end{eqnarray}
 From Lemma \ref{action_center}  for $j=\ell$,  we have that,  for $1\leq r  \leq \ell -1$ and $ r < s \leq \ell -1$, 
$[ H_\ell, \fr{m}_{r s}] =( 0 )$ and that,  $ \ell +1 \leq r < s$,  $[ H_\ell, \fr{m}_{r s}] =( 0 )$. Hence, we see that the first part  and fifth sum vanishes.  From Lemma \ref{action_center}  for for $j= t  >  \ell$,  we have, for $r \leq \ell$, $ [H_t, \fr{m}_{r \ell}]=(0)$. Hence, we see the second part vanishes. 

Note that $ [H_t, \fr{m}_{r, s}]=(0)$ for $s \leq  t-1$ and  $[H_t, \fr{m}_{\ell, s}]=(0)$ for $s \leq  t-1$ from Lemma \ref{action_center}.  
Thus  we see that equation (\ref{ell}) becomes 
\begin{eqnarray}
 && 0= \sum_{r \leq \ell-1,\, t \leq s}\frac{1}{{x_{r \, s}}^2} \Big( \sum_{\al, \be} \left(B( \widetilde{Y}_{\al \be}^{(r \, s)}, [H_{\ell}, \widetilde{X}_{\al \be}^{(r \, s )}]) B( \widetilde{Y}_{\al \be}^{(r \, s)}, [H_t, \widetilde{X}_{\al \be}^{(r \, s)}]) \right. \nonumber 
 \\
 && \left. 
 +  B( \widetilde{X}_{\al \be}^{(r\, s)}, [H_{\ell}, \widetilde{Y}_{\al \be}^{(r \, \ell )}]) B( \widetilde{X}_{\al \be}^{(r \, s)}, [H_t, \widetilde{Y}_{\al \be}^{(r \, s )}])\right)\Big) \label{eq14} 
\\
&& + \sum_{  t  \leq s}\frac{1}{{x_{ \ell  \, s}}^2} \Big( \sum_{\al, \be} \left(B( \widetilde{Y}_{\al \be}^{(\ell \, s)}, [H_{\ell}, \widetilde{X}_{\al \be}^{(\ell \, s)}]) B( \widetilde{Y}_{\al \be}^{(\ell \, s)}, [H_t, \widetilde{X}_{\al \be}^{(\ell \, s)}]) \right. \nonumber 
 \\&& \left. 
 +  B( \widetilde{X}_{\al \be}^{(\ell  \, s1)}, [H_{\ell }, \widetilde{Y}_{\al \be}^{(\ell  \, s)}]) B( \widetilde{X}_{\al \be}^{(\ell  \, s )}, [H_t, \widetilde{Y}_{\al \be}^{(\ell  \, s)}])\right)\Big). \label{eq15} 
 \end{eqnarray}
From Lemma \ref{action_center} and  equations (\ref{eqx_{ij}})
hold for $i \leq \ell-1$, we see that part  (\ref{eq14}) becomes 
\begin{eqnarray}
&&\sum_{r \leq \ell-1}\frac{1}{{x_{r \, t}}^2} \Big( \sum_{\al, \be} \left(B( \widetilde{Y}_{\al \be}^{(r \, t)}, [H_{\ell}, \widetilde{X}_{\al \be}^{(r \, t )}]) B( \widetilde{Y}_{\al \be}^{(r \, t)}, [H_t, \widetilde{X}_{\al \be}^{(r \, t)}]) \right. \nonumber
\\
 && \left. 
 +  B( \widetilde{X}_{\al \be}^{(r\, t)}, [H_{\ell}, \widetilde{Y}_{\al \be}^{(r \, t)}]) B( \widetilde{X}_{\al \be}^{(r \, t)}, [H_t, \widetilde{Y}_{\al \be}^{(r \, t)}])\right)\Big) \nonumber
\\ 
 &&  +\sum_{r \leq \ell-1,\, t < s}\frac{1}{{x_{r \, s}}^2} \Big( \sum_{\al, \be} \left(B( \widetilde{Y}_{\al \be}^{(r \, s)}, [H_{\ell}, \widetilde{X}_{\al \be}^{(r \, s )}]) B( \widetilde{Y}_{\al \be}^{(r \, s)}, [H_t, \widetilde{X}_{\al \be}^{(r \, s)}]) \right. \nonumber 
 \\
 && \left. 
 +  B( \widetilde{X}_{\al \be}^{(r\, s)}, [H_{\ell}, \widetilde{Y}_{\al \be}^{(r \, \ell )}]) B( \widetilde{X}_{\al \be}^{(r \, s)}, [H_t, \widetilde{Y}_{\al \be}^{(r \, s )}])\right)\Big)  \nonumber 
 \\&& =\sum_{r \leq \ell-1}\frac{1}{{x_{r \, t}}^2} \Big(\frac{1}{N-K_\ell}\Big(-\frac{1}{k_t} \Big) 2 k_r  k_t\Big)+\sum_{t <s} \frac{1}{{x_{r \, s}}^2} \Big(\frac{1}{N-K_\ell}\Big(\frac{1}{N-K_t} \Big) 2 k_r  k_s\Big)
  \nonumber 
  \\&&
  = \frac{1}{N-K_\ell}\sum_{r \leq \ell-1}\frac{1}{{x_{r \, t}}^2} \Big(- 2 k_r   +\sum_{t <s}  \frac{1}{N-K_t}   2 k_r  k_s\Big) =0.    \nonumber 
  \end{eqnarray}
  From Lemma \ref{action_center}, we see that  part (\ref{eq15}) becomes 
\begin{eqnarray}
&& \sum_{t  \leq s}\frac{1}{{x_{ \ell  \, s}}^2} \Big( \sum_{\al, \be} \left(B( \widetilde{Y}_{\al \be}^{(\ell \, s)}, [H_{\ell}, \widetilde{X}_{\al \be}^{(\ell \, s)}]) B( \widetilde{Y}_{\al \be}^{(\ell \, s)}, [H_t, \widetilde{X}_{\al \be}^{(\ell \, s)}]) \right. \nonumber 
\\ & &
 \left. 
 +  B( \widetilde{X}_{\al \be}^{(\ell  \, s)}, [H_{\ell }, \widetilde{Y}_{\al \be}^{(\ell  \, s)}]) B( \widetilde{X}_{\al \be}^{(\ell  \, s )}, [H_t, \widetilde{Y}_{\al \be}^{(\ell  \, s)}])\right)\Big) \nonumber 
 \end{eqnarray}
\begin{eqnarray}
&& =\frac{1}{{x_{\ell \, t}}^2} \Big(\frac{1}{k_\ell}+\frac{1}{N-K_\ell}\Big)\Big(-\frac{1}{k_t} \Big) 2 k_\ell  k_t\Big)+\sum_{t <s} \frac{1}{{x_{\ell \, s}}^2} \Big(\frac{1}{k_\ell}+\frac{1}{N-K_\ell}\Big)\Big(\frac{1}{N-K_t} \Big) 2 k_\ell  k_s\Big) \nonumber 
\\&& =\Big(\frac{1}{k_\ell}+\frac{1}{N-K_\ell}\Big) 2 k_\ell \Big(-\frac{1}{{x_{\ell \, t}}^2}+ \sum_{t <s} \frac{1}{{x_{\ell \, s}}^2}\frac{k_s}{N-K_t}\Big).  \nonumber 
  \end{eqnarray}
  Therefore, we obtain that  $ r(H_{\ell}, H_t) =0$ $(\ell  <  t)$ if and only if  
  \begin{eqnarray}\label{eq16}
  -\frac{1}{{x_{\ell \, t}}^2}+ \sum_{t <s} \frac{1}{{x_{\ell \, s}}^2}\frac{k_s}{N-K_t}.
     \end{eqnarray} 
For $t=p-1$ we see that $0= \displaystyle - \frac{1}{x_{\ell \,  p-1}^2} 
 + \frac{1}{N-K_{p-1}} \frac{1}{x_{\ell\, p}^2} k_p = - \frac{1}{x_{\ell \,  p-1}^2} 
 +  \frac{1}{x_{\ell \,  p}^2}$, and hence, $x_{\ell \, p-1} = x_{\ell \,  p}$. From (\ref{eq16}), we see our claim  for $i=\ell$ by induction. 
\end{proof}
We denote by  $\{\widetilde{H}_1, \widetilde{H}_2, \dots, \widetilde{H}_{p-1} \}$  the orthonormal basis of  the center $\fr{c}$ obtained from the orthogonal basis $\{H_1, H_2, \dots, H_{p-1} \}$.

\begin{lemma}\label{i rs} 
 For $i=1, \dots , p-1$, we have the following: 
\begin{equation*}
\begin{array}{l} 
\vspace{0.3cm}
1)  \ \  \mbox{For} \quad r < s <  i, \quad \displaystyle{ {\fr{m}_{r s} \brack {\fr{c}_{i} \ \fr{m}_{r s}}} = 0}. 
\\  \vspace{0.3cm}
2) \ \ \mbox{For} \quad r  \leq i-1, \quad \displaystyle{ {\fr{m}_{r i} \brack {\fr{c}_{i} \ \fr{m}_{r i}}} = \frac{N-K_{i}}{N(N-K_{i-1})}k_r }.   \quad
\\  \vspace{0.3cm}
3) \ \ \mbox{For} \quad  r  \leq i-1 \ \mbox{and} \   s  \geq i+1, \quad \displaystyle{ {\fr{m}_{r s} \brack {\fr{c}_{i} \ \fr{m}_{r s}}} = \frac{k_i k_r  k_s}{N(N-K_{i-1}) (N- K_{i})} }.  \quad  
\\\vspace{0.3cm}
4 ) \ \ \mbox{For} \quad  s  \geq i+1,  \ \displaystyle{ {\fr{m}_{1 s} \brack {\fr{c}_{1} \ \fr{m}_{1 s}}} = \frac{1}{N-K_{1}}k_s }, \ \ \displaystyle{ {\fr{m}_{i s} \brack {\fr{c}_{i} \ \fr{m}_{i s}}} = \frac{N-K_{i-1}}{N(N-K_{i})}k_s } \ \  \mbox{for}  \ \ i \geq 2. \\
5) \ \  \mbox{For} \quad  s > r \geq i+1,  \ \displaystyle{ {\fr{m}_{r s} \brack {\fr{c}_{i} \ \fr{m}_{r s}}} = 0}. 
\end{array} 
\end{equation*}

In particular, if $k_2 = \cdots = k_p = k$, we have 
\begin{equation*}
\begin{array}{l} 
\vspace{0.3cm}
6) \ \  \displaystyle{ {\fr{m}_{1 i} \brack {\fr{c}_{i} \ \fr{m}_{1 i}}} = 
   \frac{p - i}{N(p - i +1)}k_1 } \ \ \mbox{for} \ \  2 \leq i , 
   \ \    \displaystyle{ {\fr{m}_{r i} \brack {\fr{c}_{i} \ \fr{m}_{r i}}} = 
   \frac{p - i}{N(p - i + 1)}k} \ \ \mbox{for} \ \  2 \leq r \leq  i-1, 
\\ \vspace{0.3cm}
7) 
 \ \ \displaystyle{ {\fr{m}_{1 s} \brack {\fr{c}_{i} \ \fr{m}_{1 s}}} = \frac{k_1}{N(p - i + 1) (p- i)}}      \ \ \mbox{for} \quad  2 \leq i \ \mbox{and} \   s  \geq i+1, \\
   \vspace{0.3cm} \quad  \ \ \displaystyle{ {\fr{m}_{r s} \brack {\fr{c}_{i} \ \fr{m}_{r s}}} = \frac{k}{N(p - i + 1) (p- i)}} \ \ \mbox{for} \ \  2 \leq r \leq  i-1 \ \mbox{and} \   s  \geq i+1,  
\\\vspace{0.3cm}
8) \ \   \displaystyle{ {\fr{m}_{1 s} \brack {\fr{c}_{1} \ \fr{m}_{1 s}}} = \frac{1}{p-1} }  \quad  \mbox{for} \quad  s  \geq 2,
 \ \     \displaystyle{ {\fr{m}_{i s} \brack {\fr{c}_{i} \ \fr{m}_{i s}}} =
   \frac{p - i +1}{N(p - i )}k}  \ \  \mbox{for}  \ \ i \geq 2 \ \mbox{and} \   s  \geq i+1. 
\end{array} 
\end{equation*}
\end{lemma}
\begin{proof}  For $i = 1, \dots, p-1$, put $\widetilde{H}_i = a_i H_i$. Then we see 
$$1 = B(\widetilde{H}_i, \widetilde{H}_i) = {a_i}^2  (-2 N \tr (H_i H_i)) 
= {a_i}^2 \,  2 N \left(\frac{1}{k_i}+ \frac{1}{N-K_i}\right). $$
Thus we have $\displaystyle a_i= \sqrt{\frac{k_i (N-K_{i})}{2 N(N-K_{i-1})}}$. 

For  $r  \leq i-1$,  we have 
\begin{eqnarray*}  
 {\fr{m}_{r i} \brack {\fr{c}_{i} \ \fr{m}_{r i}}} =
\frac{k_i (N-K_{i})}{2 N(N-K_{i-1})}\left( \frac{1}{k_i}\right)^2 \, 2 k_r k_i  = \frac{N-K_{i}}{N(N-K_{i-1})}k_r .
\end{eqnarray*}  

For $ r  \leq i-1 $ and  \  $ s  \geq i+1$, we have 
\begin{eqnarray*} 
 {\fr{m}_{r s} \brack {\fr{c}_{i} \ \fr{m}_{r s}}} = \frac{k_i (N-K_{i})}{2 N(N-K_{i-1})} \left( \frac{1}{N-K_i}\right)^2 \, 2 k_r k_s }$ $\displaystyle{ =\frac{k_i k_r  k_s}{N(N-K_{i-1}) (N- K_{i})}. \end{eqnarray*}  

For $i=1, \,  r  =1 $ and  \  $ s  \geq i+1$, we have 
\begin{eqnarray*}
 {\fr{m}_{1 s} \brack {\fr{c}_{1} \ \fr{m}_{1 s}}} = \frac{k_1(N-K_{1})}{2 N^2} \left(\frac{1}{k_1}+\frac{1}{N - K_1}\right)^2 \, 2 k_1 k_s =\frac{ k_s}{(N- K_{1})}. 
 \end{eqnarray*}  

For $i \geq 2$ and  \  $ s  \geq i+1$, we have 
\begin{eqnarray*} {\fr{m}_{i s} \brack {\fr{c}_{i} \ \fr{m}_{i s}}} = \frac{k_i(N-K_{i})}{2 N(N-K_{i-1})} \left( \frac{1}{k_i}+\frac{1}{N - K_i}\right)^2 \, 2 k_i k_s 
 =  \frac{N-K_{i-1}}{N(N-K_{i})}k_s. 
 \end{eqnarray*}

Now 6), 7) and 8) follow from 2), 3) and 4). 
\end{proof}

\begin{lemma}\label{k k_1} 
For $k_2 = \cdots = k_p = k$, 
we have the following: 
\begin{equation*}
\begin{array}{l} 
(1) \ \, \displaystyle{ {\fr{m}_{1 s} \brack {\fr{c}_{1} \ \fr{m}_{1 s}}} = \frac{1}{p-1}} \quad  \mbox{for} \quad  s  \geq 2,  \quad   \displaystyle{ \sum_{2 \leq i}{\fr{m}_{1 s} \brack {\fr{c}_{i} \ \fr{m}_{1 s}}} = \frac{(p-2) k_1}{(p-1) N} } \quad  \mbox{for} \quad  s  \geq 2, \\ 
(2) \ \, \displaystyle{ \sum_{ i }{\fr{m}_{r s} \brack {\fr{c}_{i} \ \fr{m}_{r s}}} = \sum_{ i \geq r }{\fr{m}_{r s} \brack {\fr{c}_{i} \ \fr{m}_{r s}}} = \frac{2 k}{ N} } \quad  \mbox{for} \quad  2  \leq r  < s,\\
(3) \ \,\displaystyle{  \sum_{2 \leq s}{\fr{m}_{1 s} \brack {\fr{c}_{t} \ \fr{m}_{1 s}}} = \frac{  k_1}{N} } \quad  \mbox{for} \quad  t \geq 2, \quad  \displaystyle{  \sum_{2 \leq r < s}{\fr{m}_{r s} \brack {\fr{c}_{t} \ \fr{m}_{r s}}} = \frac{(p-1)  k}{N} }  \quad  \mbox{for} \quad  t  \geq 2.
\end{array} 
\end{equation*}
\end{lemma}
\begin{proof} 
(1) \ Note that  $\displaystyle{ {\fr{m}_{1 s} \brack {\fr{c}_{i} \ \fr{m}_{1 s}}} = \frac{k_1}{N(p - i + 1) (p- i)} }$ 
 $\displaystyle{=   \frac{k_1}{N}\left(\frac{1}{p- i}+\frac{-1}{p - i + 1} \right)}$ for $2 \leq i \leq s-1$  from  7) of Lemma \ref{i rs},  $\displaystyle{ {\fr{m}_{1 s} \brack {\fr{c}_{s} \ \fr{m}_{1s}}} = 
   \frac{p - s}{N(p - s +1)}k_1 }$  for $s \leq p-1$ from  6) of Lemma \ref{i rs} and $\displaystyle{ {\fr{m}_{1 s} \brack {\fr{c}_{i} \ \fr{m}_{1 s}}}=0}$ for $i \geq s+1$. 
  Thus we see that, for $s \leq p-1$,  
\begin{eqnarray*}  \sum_{2 \leq i}{\fr{m}_{1 s} \brack {\fr{c}_{i} \ \fr{m}_{1 s}}} } \displaystyle{= \frac{k_1}{N}\left(\frac{-1}{ p-1} + \frac{1}{p-s+1}+\frac{p-s}{ p-s+1 } \right) = \frac{(p-2) k_1}{(p-1) N}.
   \end{eqnarray*}
  For $s=p$, we see that 
  \begin{eqnarray*}  \sum_{2 \leq i \leq p-1}{\fr{m}_{1 p} \brack {\fr{c}_{i} \ \fr{m}_{1 p}}} } \displaystyle{= \frac{k_1}{N}\left(\frac{1}{(p - 1) (p-2)} + \cdots + \frac{1}{3 \cdot 2}+\frac{1}{ 2 \cdot 1 } \right) = \frac{(p-2) k_1}{(p-1) N}.
  \end{eqnarray*}
\ \  (2) \ For $2  \leq r$,  we have $ \displaystyle{{\fr{m}_{r s} \brack {\fr{c}_{i} \ \fr{m}_{r s}}} } =0$ for $1 \leq i \leq r-1$ from  Lemma \ref{i rs} (5) and $ \displaystyle{{\fr{m}_{r s} \brack {\fr{c}_{i} \ \fr{m}_{r s}}} } =0$ for $s+1 \leq i $ from  Lemma \ref{i rs} (1). 
Thus we have 
\begin{eqnarray*}  && \sum_{ i }{\fr{m}_{r s} \brack {\fr{c}_{i} \ \fr{m}_{r s}}} = 
\sum_{r \leq i \leq s}{\fr{m}_{r s} \brack {\fr{c}_{i} \ \fr{m}_{r s}}} = \frac{k}{N}\left(
\frac{p-r+1}{p-r}  +\frac{1}{(p-r)(p-(r+1))} \right.\\ && \left. +\cdots +\frac{1}{(p-s+2)(p-s+1)}+\frac{p-s}{p-s+1}\right)=  \frac{ 2 k}{N} 
\end{eqnarray*}
 for $s \leq p-1$. For $s=p$, we have 
 \begin{eqnarray*}   && \sum_{ i }{\fr{m}_{r p} \brack {\fr{c}_{i} \ \fr{m}_{r p}}} = 
\sum_{r \leq i \leq p-1}{\fr{m}_{r p} \brack {\fr{c}_{i} \ \fr{m}_{r p}}}  \\&& = \frac{k}{N}\left(
\frac{p-r+1}{p-r}+\frac{1}{(p-r)(p-(r+1))} +\cdots \right. }$ $\displaystyle{\left.  +\frac{1}{2\cdot 1}\right) = \frac{ 2 k}{N}.\end{eqnarray*}

 (3) \  From  1) of Lemma \ref{i rs},  we have 
  \begin{eqnarray*} &&
 \sum_{2 \leq s}{\fr{m}_{1 s} \brack {\fr{c}_{t} \ \fr{m}_{1 s}}} = {\fr{m}_{1 t} \brack {\fr{c}_{t} \ \fr{m}_{1 t}}} + \sum_{t+1 \leq s}{\fr{m}_{1 s} \brack {\fr{c}_{t} \ \fr{m}_{1 s}}} \\ 
 && = \frac{  k_1}{N}\left(\frac{ p-t}{p-t+1} +\frac{ 1}{(p-t+1)(p-t)} (p-t)\right) =  \frac{  k_1}{N}.
 \end{eqnarray*}   
 From  1) and 5) of Lemma \ref{i rs},  we have 
 \begin{eqnarray*} 
 & &  \sum_{2 \leq r < s}{\fr{m}_{r s} \brack {\fr{c}_{t} \ \fr{m}_{r s}}} =  \sum_{2 \leq r \leq t-1}{\fr{m}_{r t} \brack {\fr{c}_{t} \ \fr{m}_{r t}}} + \sum_{\substack{t+1 \leq s, \\ 2 \leq r \leq t-1} }{\fr{m}_{r s} \brack {\fr{c}_{t} \ \fr{m}_{r s}}}  +\sum_{t+1 \leq s \leq p}{\fr{m}_{t s} \brack {\fr{c}_{t} \ \fr{m}_{t s}}  }\\& &
   = \frac{  k}{N}\left(\frac{ p-t}{p-t+1}(t-2) +\frac{ 1}{(p-t+1)(p-t)} (p-t)(t-2)+ \frac{ p-t+1}{p-t}(p-t)\right) \\ & &=  \frac{  k}{N}(p-2). 
\end{eqnarray*}   
\end{proof} 

\noindent
 The following lemma generalizes Lemma 3.3 in \cite{ASS2}.

\begin{lemma}\label{ijk}   
We have the following: 
\begin{equation*}
\begin{array}{l} 
(1) \ \, \displaystyle{ {\fr{m}_{r s} \brack {\fr{m}_{r t} \ \fr{m}_{t s}}} = \frac{k_{r}k_{s}k_{t}}{N}},   \quad  \ \
(2) \ \,\displaystyle{ {\fr{m}_{i j} \brack {\fr{m}_{i} \ \fr{m}_{i j}}} = \frac{k_{j} ({k_{i}}^2-1)}{N} }, \quad   \ \
(3) \ \,\displaystyle{ {\fr{m}_{i} \brack {\fr{m}_{i} \ \fr{m}_{i}}} = \frac{k_{i} ({k_{i}}^2-1)}{N} }.
\end{array} 
\end{equation*}
\end{lemma}
\begin{proof}  It follows from an easy adjustment of the proof of \cite[Lemma 3.3]{ASS2} and we omit it. 
\end{proof}

\subsection{A class of invariant metrics}
Now we consider  the special case of the metric (\ref{metric_g})  on $\SU(k_1+(p-1) k)$ with $x_{1 j} = x_{1 2}$ for $j= 2, \ldots, p$, $x_{23} =  x_{r s}$ for $2 \leq r < s \leq p$, $x_j = x_2$  for $j= 2, \ldots, p$ and $y_j = y_2$  for $j= 2, \ldots, p-1$: 
 \begin{equation}\label{metric_g111} 
 g = y_1 B |_{ \fr{c}_1 }+ y_2\sum_{j=2}^{p-1} B |_{ \fr{c}_j }+  x_1 B |_{\fr{m}_1}+x_2\sum_{\ell=2}^{p} B |_{\fr{m}_\ell}+  x_{12}\sum_{2\le s \le p}  B |_{ \fr{m}_{1 s}}+ x_{2 3}\sum_{2\le r<s \le p}  B |_{ \fr{m}_{r s}}. 
 \end{equation}
 
 We introduce the following notation for Ricci  components of the metric (\ref{metric_g111}):
 \begin{eqnarray*}
 rr_i &=&\Ric(g)(\widetilde{H}_i, \tilde{H}_i),\qquad i=1,\dots,  p-1\\
 r_j &=& \Ric(g)(\widetilde{X}_i, \widetilde{X}_i), \qquad  j=1,\dots,  p\\
 r_{r s} &=& \Ric(g)(\widetilde{X}_{\al\be}^{(rs)}, \widetilde{X}_{\al\be}^{(rs)}), \qquad 1\leq r < s \leq  p.
 \end{eqnarray*}
 Here
 $\widetilde{X}_i\in\fr{m}_{i}$ and $\widetilde{X}_{\al\be}^{(rs)}\in\fr{m}_{r s}$.

\begin{prop}\label{ric_component}
For the Ricci components of the metric {\em(\ref{metric_g111}) }, we have that 
\begin{equation*} 
\begin{array}{llll} 
rr_2^{} = \cdots = rr_{p-1}^{}, &  r_2^{} = \cdots =  r_p^{},  & r_{1 s} = r_{1 2},  & r_{2 3} = r_{r s} \  ( 2\leq r < s \leq  p). 
\end{array} 
\end{equation*}

Moreover, we have 
\begin{eqnarray}
  rr_1^{} &=& \frac{1}{4} \frac{y_1}{{x_{12}}^2}, \\ 
  rr_2^{} &= & \frac{k_1}{4 N} \frac{y_2}{{x_{12}}^2} + \frac{k}{4 N} (p-1)\frac{y_2}{{x_{23}}^2}, \\
  r_1^{} &= & \frac{k_1}{4 N} \frac{1}{x_1}+\frac{k}{4 N}(p-1)\frac{x_1}{{x_{12}}^2}, \\ 
  r_2^{} &= &  \frac{k}{4 N} \frac{1}{x_2} +\frac{k_1}{4 N} \frac{x_2}{{x_{12}}^2}+ \frac{k}{4 N} (p-2) \frac{x_2}{{x_{23}}^2}, \\
  r_{12}& = & \frac{1}{2 x_{12}} - \frac{k}{4 N}(p-2)\frac{x_{23}}{{x_{12}}^2}
  - \frac{1}{4 N}\left(\frac{{k_1}^2-1}{k_1}\frac{x_1}{{x_{12}}^2}+ \frac{{k}^2-1}{k}\frac{x_2}{{x_{12}}^2}\right) \nonumber \\
  & & - \frac{1}{4 k k_1 (p-1)}\left(\frac{y_1}{x_{12}^2}+ (p-2)\frac{k_1}{N}\frac{y_2}{{x_{12}}^2}\right),\\
  r_{23}& = & \frac{1}{2 x_{23}}+ \frac{k_1}{4 N}\left(\frac{x_{23}}{{x_{12}}^2}- \frac{2}{x_{23}}\right) -(p-3)\frac{k}{4 N }\frac{1}{x_{23}}- \frac{{k}^2-1}{2 k N}\frac{x_2}{{x_{23}}^2}- \frac{1}{2 k N}\frac{y_2}{{x_{23}}^2}. 
\end{eqnarray}
\end{prop}
\begin{proof} From Lemma \ref{r(H_i, H_j)}, we see that, for $ 1 \leq i \leq p-1$, 
\begin{eqnarray*} rr_i= \frac{1}{4}\sum_{r < s}\frac{y_i}{{x_{r s}}^2} {\fr{m}_{r s} \brack {\fr{c}_{i} \ \fr{m}_{r s}}}.  
\end{eqnarray*} 
 For $i=1$,  we have, from  5) and 8) of Lemma \ref{i rs}, 
\begin{eqnarray*} rr_1= \frac{1}{4}\sum_{1 < s}\frac{y_1}{{x_{1 s}}^2} {\fr{m}_{1 s} \brack {\fr{c}_{1} \ \fr{m}_{1 s}}} = \frac{1}{4}\sum_{1 < s}\frac{y_1}{{x_{1 s}}^2} \frac{1}{p-1} = 
\frac{1}{4}\frac{y_1}{{x_{1 2}}^2} \frac{1}{p-1} (p-1)= \frac{1}{4}\frac{y_1}{{x_{1 2}}^2}. 
\end{eqnarray*} 
For $2 \leq i \leq p-1$,  from  (3)  of Lemma \ref{k k_1},  we have 
\begin{eqnarray*} && rr_i= 
\frac{1}{4}\sum_{r < s}\frac{y_i}{{x_{r s}}^2} {\fr{m}_{r s} \brack {\fr{c}_{i} \ \fr{m}_{r s}}} = \frac{1}{4}\left(\sum_{s>1}\frac{y_i}{{x_{1 s}}^2} {\fr{m}_{1 s} \brack {\fr{c}_{i} \ \fr{m}_{1 s}}}+ \sum_{2 \leq r < s}\frac{y_i}{{x_{r s}}^2} {\fr{m}_{r s} \brack {\fr{c}_{i} \ \fr{m}_{r s}}}\right)\\
&& =\frac{1}{4}\left(\frac{y_i}{{x_{1 2}}^2}\sum_{1 < s} {\fr{m}_{1 s} \brack {\fr{c}_{i} \ \fr{m}_{1 s}}}+ \frac{y_i}{{x_{23}}^2} \sum_{2 \leq r < s}{\fr{m}_{r s} \brack {\fr{c}_{i} \ \fr{m}_{r s}}}\right) = \frac{1}{4}\left( \frac{y_2}{{x_{1 2}}^2} \frac{k_1}{N} + \frac{y_2}{{x_{23}}^2} \frac{(p-1)k}{N}\right).  
\end{eqnarray*} 
We have 
\begin{eqnarray*}&& \hspace{-12pt} r_i= \frac{1}{2 x_i}+ \frac{1}{4 d_i} \left( \frac{1}{x_i}{\fr{m}_{i} \brack {\fr{m}_{i} \ \fr{m}_{i}}}+\sum_{s \neq i}\frac{x_i}{{x_{i s}}^2} {\fr{m}_{i} \brack {\fr{m}_{i s} \ \fr{m}_{i s}}}\right) - \frac{1}{2 d_i} \left( \frac{1}{x_i}{\fr{m}_{i} \brack {\fr{m}_{i} \ \fr{m}_{i}}}+\sum_{s \neq i}\frac{1}{x_i} {\fr{m}_{i} \brack {\fr{m}_{i s} \ \fr{m}_{i s}}}\right) \\
&& = \frac{1}{2 x_i} + \frac{1}{4 d_i} \left( -\frac{1}{x_i}{\fr{m}_{i} \brack {\fr{m}_{i} \ \fr{m}_{i}}}-\sum_{s \neq i}\frac{2}{x_i} {\fr{m}_{i} \brack {\fr{m}_{i s} \ \fr{m}_{i s}}}+ \sum_{s \neq i}\frac{x_i}{{x_{i s}}^2} {\fr{m}_{i} \brack {\fr{m}_{i s} \ \fr{m}_{i s}}}\right).  
\end{eqnarray*} 
For $i=1$, from Lemma \ref{ijk} (2) and (3), we see that 
\begin{eqnarray*}&&  r_1=\frac{1}{2 x_1} + \frac{1}{4( {k_1}^2 -1)} \left( -\frac{1}{x_1} \frac{k_1({k_1}^2-1)}{N}-\sum_{s > 1}\frac{2}{x_1} \frac{ k ({k_1}^2-1)}{N}+ \sum_{s >1}\frac{x_1}{{x_{1 s}}^2} \frac{ k ({k_1}^2-1)}{N}\right) \\ &&
=\frac{1}{4 x_1}\left( 2 -\frac{k_1+ 2 (p-1)k}{N }\right) + \frac{(p-1)k}{4 N}\frac{x_1}{{ x_{12}}^2} = \frac{k_1}{4 N} \frac{1}{x_1}+  \frac{(p-1)k}{4 N}\frac{x_1}{{ x_{12}}^2}. 
\end{eqnarray*}  
For  $i\geq 2$, from Lemma \ref{ijk} (2) and (3), we see that 
\begin{eqnarray*}&&  r_i=\frac{1}{2 x_i} + \frac{1}{4( {k}^2 -1)} \left( -\frac{1}{x_i} \frac{k({k}^2-1)}{N}-\frac{2}{x_i} \frac{ k_1 ({k}^2-1)}{N}+ \frac{x_i}{{x_{1 i}}^2} \frac{ k_1 ({k}^2-1)}{N}\right.  \\ && \left.
-\sum_{s \neq 1,  i}\frac{2}{x_i} \frac{ k ({k}^2-1)}{N}+ \sum_{s \neq  1, i}\frac{x_i}{{x_{i s}}^2} \frac{ k ({k}^2-1)}{N}\right)
 = \frac{k}{4N x_2} +\frac{k_1}{4N} \frac{x_2} {{x_{12} }^2}+ \frac{k}{4N} (p-2)\frac{x_2} {{x_{23} }^2}. 
\end{eqnarray*}
We have 
\begin{eqnarray*}&&  r_{i j} = \frac{1}{2 x_{i j}} + \frac{1}{2 d_{i j}}\sum_{t \neq i, j} {\fr{m}_{i j} \brack {\fr{m}_{i t} \ \fr{m}_{j t}}}\left( \frac{x_{i j}}{x_{i t} x_{t j} }-  \frac{x_{i t}}{x_{i j} x_{t j}} -  \frac{x_{t j}}{x_{i j} x_{i t}} \right) + \frac{1}{4 d_{i j}}\left( {\fr{m}_{i j} \brack {\fr{m}_{i } \ \fr{m}_{i j}}}\frac{1}{x_i} \times 2 \right. \\
&&\left.+ {\fr{m}_{i j} \brack {\fr{m}_{j } \ \fr{m}_{i j}}}\frac{1}{x_j} \times 2 + \sum_{\ell}{\fr{m}_{i j} \brack {\fr{c}_{\ell} \ \fr{m}_{i j}}}\frac{1}{y_\ell} \times 2 \right) - \frac{1}{2 d_{i j}}\left( {\fr{m}_{i j} \brack {\fr{m}_{i j } \ \fr{m}_{i}}}\frac{1}{x_i}+ {\fr{m}_{i } \brack {\fr{m}_{i j }  \fr{m}_{ij}}}\frac{x_i}{{x_{ij}}^2} \right. \\ &&
\left.+ {\fr{m}_{i j} \brack {\fr{m}_{i j } \ \fr{m}_{j}}}\frac{1}{x_j}+ {\fr{m}_{j } \brack {\fr{m}_{i j } \ \fr{m}_{ij}}}\frac{x_j}{{x_{ij}}^2}+ \sum_{\ell}{\fr{m}_{i j} \brack {\fr{c}_{\ell} \ \fr{m}_{i j}}}\frac{1}{y_\ell} + \sum_{\ell}{\fr{m}_{i j} \brack {\fr{c}_{\ell} \ \fr{m}_{i j}}}\frac{1}{y_\ell} \right). 
\end{eqnarray*}
For $i=1$, from Lemma \ref{ijk} (1), (2) and Lemma \ref{k k_1} (1),   we see that 
\begin{eqnarray*}&&  r_{1 j} =   \frac{1}{2 x_{1j}} + \frac{1}{4 k_1 k }\sum_{t \neq 1, j} {\fr{m}_{1 j} \brack {\fr{m}_{1 t} \ \fr{m}_{j t}}}\left( \frac{x_{1 j}}{x_{1 t} x_{t j} }-  \frac{x_{1 t}}{x_{1j} x_{t j}} -  \frac{x_{t j}}{x_{1 j} x_{1 t}} \right) + \frac{1}{8 k_1 k }\left( {\fr{m}_{1 j} \brack {\fr{m}_{1} \ \fr{m}_{1 j}}}\frac{1}{x_1} \times 2 \right. \\
&&\left.+ {\fr{m}_{1 j} \brack {\fr{m}_{j } \ \fr{m}_{1 j}}}\frac{1}{x_j} \times 2 + \sum_{\ell}{\fr{m}_{1 j} \brack {\fr{c}_{\ell} \ \fr{m}_{1 j}}}\frac{1}{y_\ell} \times 2 \right) - \frac{1}{4 k_1 k }\left( {\fr{m}_{1 j} \brack {\fr{m}_{1 j } \ \fr{m}_{1}}}\frac{1}{x_1}+ {\fr{m}_{1 } \brack {\fr{m}_{1 j }  \fr{m}_{1j}}}\frac{x_1}{{x_{1j}}^2} \right. \\ &&
\left.+ {\fr{m}_{1 j} \brack {\fr{m}_{1 j } \ \fr{m}_{j}}}\frac{1}{x_j}+ {\fr{m}_{j } \brack {\fr{m}_{1 j } \ \fr{m}_{1 j}}}\frac{x_j}{{x_{1 j}}^2}+ \sum_{\ell}{\fr{m}_{1 j} \brack {\fr{c}_{\ell} \ \fr{m}_{1 j}}}\frac{1}{y_\ell} + \sum_{\ell} {\fr{c}_{\ell} \brack {\fr{m}_{1 j}  \ \fr{m}_{1 j}}}\frac{y_\ell}{{x_{1 \ell}}^2} \right) \\
&&  = \frac{1}{2 x_{12}} -(p-2)\frac{ k}{ 4N}\frac{x_{23}}{{x_{12}}^2}
\\&& - \left( \frac{{k_1}^2-1}{4 k_1 N}\frac{x_1}{{x_{12}}^2}+  \frac{ {k}^2-1}{ 4 k N}\frac{x_2}{{x_{1 2}}^2} +\frac{1}{4 k_1 k (p-1)}\frac{y_1}{{x_{12}}^2}
+ \frac{(p-2)}{ 4  k (p-1) N} \frac{y_2}{{x_{12}}^2}\right). 
\end{eqnarray*}
For $2 \leq i < j$, from Lemma \ref{ijk} (1), (2) and Lemma \ref{k k_1} (2),   we see that 
\begin{eqnarray*}&&  r_{i j} =   \frac{1}{2 x_{i j}} + \frac{1}{4 k^2 }\sum_{t \neq i, j} {\fr{m}_{i j} \brack {\fr{m}_{i t} \ \fr{m}_{j t}}}\left( \frac{x_{i j}}{x_{i t} x_{t j} }-  \frac{x_{i t}}{x_{i j} x_{t j}} -  \frac{x_{t j}}{x_{i j} x_{i t}} \right) + \frac{1}{8  k^2 }\left( {\fr{m}_{i j} \brack {\fr{m}_{i} \ \fr{m}_{i j}}}\frac{1}{x_i} \times 2 \right. \\
&&\left.+ {\fr{m}_{i j} \brack {\fr{m}_{j } \ \fr{m}_{i j}}}\frac{1}{x_j} \times 2 + \sum_{\ell}{\fr{m}_{i j} \brack {\fr{c}_{\ell} \ \fr{m}_{i j}}}\frac{1}{y_\ell} \times 2 \right) - \frac{1}{4  k^2 }\left( {\fr{m}_{i j} \brack {\fr{m}_{i j } \ \fr{m}_{i}}}\frac{1}{x_i}+ {\fr{m}_{i } \brack {\fr{m}_{i j }  \fr{m}_{i j}}}\frac{x_i}{{x_{i j}}^2} \right. \\ &&
\left.+ {\fr{m}_{i j} \brack {\fr{m}_{i j } \ \fr{m}_{j}}}\frac{1}{x_j}+ {\fr{m}_{j } \brack {\fr{m}_{i j } \ \fr{m}_{i j}}}\frac{x_j}{{x_{i j}}^2}+ \sum_{\ell}{\fr{m}_{i j} \brack {\fr{c}_{\ell} \ \fr{m}_{i j}}}\frac{1}{y_\ell} + \sum_{\ell} {\fr{c}_{\ell} \brack {\fr{m}_{i j}  \ \fr{m}_{i j}}}\frac{y_\ell}{{x_{i \ell}}^2} \right) \\
&&  = \frac{1}{2 x_{23}} +\frac{ k_1}{ 4N}\left( \frac{x_{23}}{{x_{12}}^2} - \frac{2}{x_{23}}\right)
- \frac{k}{4 N}(p-3)\frac{1}{x_{23}} - \frac{ {k}^2-1}{ 2 k N}\frac{x_2}{{x_{23}} ^2}- \frac{1}{ 2 k  N} \frac{y_2}{{x_{23}}^2}. 
\end{eqnarray*}
\end{proof}

\section{Einstein metrics on the compact Lie group $\SU(k_1+ ( p -1)k)$}


To find 
 Einstein metrics on $\SU(k_1+ ( p -1)k)$ we need to solve the system
(cf. Proposition \ref{ric_component})
\begin{eqnarray}\label{system1}
& & r_{12} - r_{23} = 0,\  r_{23} - r_2 = 0, \  r_1 - r_2 = 0,\nonumber \
 r_2 - rr_1 = 0, \ rr_1 - rr_2= 0.
\end{eqnarray}
We normalize the  variables  of the equations obtained from Proposition \ref{ric_component} by setting $ x_{23} =1$. Then we have the following:
\begin{eqnarray}
  && f_1 = -k^2 {k_{1}} (p-1) (p+1)
   {x_{12}}^2+2 (k-1) (k+1) {k_{1}}
   (p-1) {x_{12}}^2 {x_{2}} \nonumber \\
   &&  \quad +2 k {k_{1}}
   (p-1) {x_{12}} (k p-k+{k_{1}})
   -k ({k_{1}}-1) ({k_{1}}+1)
   (p-1) {x_{1}}  \nonumber \\
   &&  \quad  -(k-1)
   (k+1) {k_{1}} (p-1) {x_{2}}-{y_{1}} (k
   p-k+{k_{1}}) \nonumber 
   \\ &&  \quad -k {k_{1}} (p-1) (k p-2 k+{k_{1}}) 
   +2 {k_{1}} (p-1) {x_{12}}^2 {y_{2}}-{k_{1}} (p-2) {y_{2}} =0, \label{eq24}\\
  && f_2 =  -{x_{12}}^2 {x_{2}}^2
   \left(k^2 p-2\right)+k^2 (p+1) {x_{12}}^2
   {x_{2}} 
  \nonumber \\  &&  \quad  
   -k^2 {x_{12}}^2-k {k_{1}}
   {x_{2}}^2+k {k_{1}} {x_{2}}-2
   {x_{12}}^2 {x_{2}} {y_{2}}=0, \label{eq25}\\
&&  f_3 =  k (p-1) {x_{1}}^2 {x_{2}}-k (p-2) {x_{1}} {x_{12}}^2 {x_{2}}^2-k {x_{1}} {x_{12}}^2 
    -{k_{1}} {x_{1}} {x_{2}}^2+{k_{1}} {x_{12}}^2 {x_{2}}=0, \label{eq26}\\
&&  f_4 = -{x_{2}} {y_{1}} (k p-k+{k_{1}})+k (p-2) {x_{12}}^2
   {x_{2}}^2+k {x_{12}}^2+{k_{1}}
   {x_{2}}^2=0, \label{eq27}\\
 && f_5 =  {y_{1}} (k p-k+{k_{1}})-k
   (p-1) {x_{12}}^2 {y_{2}}-{k_{1}}
   {y_{2}}=0.\label{eq28}
     \end{eqnarray}

Solving equations (\ref{eq27}) and (\ref{eq28}) for ${y_{1}}$ and ${y_{2}}$,  
we obtain 
\begin{eqnarray}
{y_1}= \frac{k(p-2) {x_{12}}^2 \, {x_{2}}^2+k
   {x_{12}}^2+{k_{1}}
   {x_{2}}^2}{x_{2} (k\,( p-1)+{k_{1}})}, \label{eq29} \\
 {y_2}= \frac{k (p-2) {x_{12}}^2
   {x_{2}}^2+k\, {x_{12}}^2+{k_{1}}
   {x_{2}}^2}{{x_{2}} \left(k \, (p-1)\, {x_{12}}^2+{k_{1}}\right)}. \label{eq30}
\end{eqnarray}
By substituting (\ref{eq29}) and (\ref{eq30}) for $y_1$ and $y_2$ 
in the  equations (\ref{eq24}) and (\ref{eq25}),  we obtain the equations  
\begin{eqnarray}
&& F_1 ={x_2}^2 \Big(-\left({x_{12}}^4
   \left(2 k^2{k_1} (p-1)-k (p-2)-2{k_1}\right)\right)+k
   {k_1}^2+k {k_1} {x_{12}}^2 (k( p-1)-2 {k_1})\Big) \nonumber \\
   &&
 +{x_2} \left(k(p-1) {x_{12}}^2+{k_1}\right)
   \Big(k {k_1} (p+1) {x_{12}}^2-2 {k_1} {x_{12}} (k(p-1)+{k_1}) +{k_1} (k (p-2)+{k_1}) \nonumber\\
   &&+({k_1}^2-1) {x_1}\Big)
 +{x_{12}}^2\left({x_{12}}^2 (k-2{k_1})+{k_1}\right)=0, \label{eq31}\\
 && F_2= ({x_2}-1) \Big({x_2}
   \big({x_{12}}^4 \left(k^2
   (p-1) p-2\right)+k {k_1} (2p-1)
   {x_{12}}^2+{k_1}^2\big) \nonumber \\
&&   - {x_{12}}^2 \left({x_{12}}^2
   \left(k^2 (p-1)+2\right)+k
   {k_1}\right)\Big) =0, \label{eq32}
\end{eqnarray}
by taking numerator of the equations.  Now we consider two cases.

\smallskip 
\noindent  
{\bf Case 1).}  \quad If $x_2=1$, we have 
\begin{eqnarray*}
{y_1} = \frac{k(p-1) {x_{12}}^2+{k_{1}}}
   {k ( p-1)+{k_{1}}},  \quad
   {y_2}= 1. 
\end{eqnarray*}
Thus the metrics are $\Ad(\s(\U(k_1)\times\U((p-1)k))$-invariant. 

From (\ref{eq26}), we see that,
$$\left({x_1}-{x_{12}}^2\right) ({k_1}-k(p-1){x_1})=0.$$

If ${x_1} = {x_{12}}^2$, from (\ref{eq31}),  we have  
$${k_1} ({x_{12}}-1)^2 (k (p-1)+{k_1}) \left(k (p-1) {x_{12}}^2+{k_1}\right) =0. $$
   Thus we see $ x_{12} =1$ and ${y_1} =1$, and hence the metric is bi-invariant. 

If $\displaystyle {x_1} = {k_1}/({k(p-1)})$, from (\ref{eq31}),  we see that $x_{12}$ is a solution of the quadratic equation $Q_1(x_{12})=0$, where 
\begin{eqnarray*}
&& Q_1(x_{12}) = k (p-1)\big(k{k_1} (p-1)+1\big) {x_{12}}^2 -2 k {k_1} (p-1) \big( {k_1}+k (p-1)\big){x_{12}}\\
&& +{k_1} \left(k^2 (p-1)^2+k{k_1} (p-1)+{k_1}^2-1\right). 
\end{eqnarray*}
Note that 
\begin{eqnarray*}
&& Q_1((k1 ( k1 + k( p-1)))/( k k1 (p-1)+1) \\
&& = - \frac{( k_1-1) k_1 (  k_1+1) ( k( p-1)-1) (k( p-1)+1)}{(1 + k k1 (p-1))} < 0 
\end{eqnarray*}
and 
$Q_1(0) = {k_1} \left(k^2 (p-1)^2+k{k_1} (p-1)+{k_1}^2-1\right) >0$. 
Thus the quadratic equation $Q_1(x_{12})=0$ has positive solutions, and hence there are two Einstein metrics. 

\smallskip 

\noindent
{\bf Case 2).} \quad  If $x_2 \neq1$, from (\ref{eq32}),   we have 
 \begin{eqnarray}
 x_2 = \frac{ {x_{12}}^2 \left({x_{12}}^2
   \left(k^2 (p-1)+2\right)+k {k_1}\right)}{{x_{12}}^4 \left(k^2
   (p-1) p-2\right)+k {k_1} (2p-1) {x_{12}}^2+{k_1}^2}. \label{eq33}
\end{eqnarray}
Note that, for a real value $ x_{12}$,  the  value $x_2$ is positive by the equation above. 
Note also that, for $k  \geq 2$ and $p \geq 3$,  
\begin{eqnarray}
x_2 < 1. 
\end{eqnarray}
In fact, we see that 
$$ 1- x_2= \frac{\left({x_{12}}^2 (k
   (p-1)-2)+k_{1}\right)
   \left({x_{12}}^2 (k
   (p-1)+2)+k_{1}\right)}{{x_{12}}^4 \left(k^2 (p-1)
   p-2\right)+k k_{1} (2 p-1)
   {x_{12}}^2+{k_{1}}^2} > 0.$$

From (\ref{eq31}), we have 
 \begin{eqnarray} 
&& {x_1} = 
  \big(-{x_2}^2
   \left(-\left({x_{12}}^4 \left(-2
   \left(k^2+1\right) {k_1}+k p(2 k{k_ 1}-1)+2 k\right)\right)+k
   {k_ 1}^2 \right. \nonumber \\
&&   \left.+k {k_1} {x_{12}}^2 (k p-k-2 {k_1})\right)-k {k_1} (p+1)
   {x_{12}}^2 {x_2} \left(k (p-1)
   {x_{12}}^2+{k_1}\right) \nonumber \\
   && +2 {k_1}
   {x_{12}}{x_2}(k (p-1)+{k_1})
   \left(k(p-1){x_{12}}^2+{k_1}\right) \nonumber \\ 
   && -{k_1} {x_2}(k(p-2)+{k_1})\left(k(p-1){x_{12}}^2+{k_1}\right)+
   {x_{12}}^2 \left(-\left({x_{12}}^2(k-2{k_1})\right)
   -{k_1}\right)\big)/\nonumber  \\
   &&  \left(({k_1}^2-1){x_2}\left(k(p-1){x_{12}}^2+{k_1}\right) \right). \label{eq34}
 \end{eqnarray}
 Thus we see that, for a real value $ x_{12}$, so  is  $x_1$. 
 
 Now we claim that, for the real solutions $x_{12}$ and $ x_2 > 0$, if 
 the quadratic equation  (\ref{eq26}) with respect to $x_1$ has  real solutions, then these  are  positive. 
 Indeed, we have 
  \begin{eqnarray*} && f_3= k (p-1) {x_{2}} \Big({x_{1}} - \frac{k (p-2) {x_{12}}^2 {x_{2}}^2 +k  {x_{12}}^2 +{k_{1}}  {x_{2}}^2}{ 2 k (p-1) {x_{2}}} \Big)^2  \\ &&+{k_{1}} {x_{12}}^2 {x_{2}} - \frac{\left(k (p-2) {x_{12}}^2 {x_{2}}^2 +k  {x_{12}}^2 
+{k_{1}}  {x_{2}}^2\right)^2}{4k (p-1) {x_{2}} }, 
  \end{eqnarray*} 
   which is quadratic for $x_1$.
 The value of $f_3 $ at $x_1=0$ is given by ${k_{1}} {x_{12}}^2 {x_{2}} >0$ and 
 $\displaystyle \frac{k (p-2) {x_{12}}^2 {x_{2}}^2 +k  {x_{12}}^2 +{k_{1}}  {x_{2}}^2}{ 2 k (p-1) {x_{2}}}  >0$ so the claim follows.

 Now we need to show that $x_{12}>0$. 
  By substituting  (\ref{eq33}) for $x_2$ and (\ref{eq34})  for $x_1$ 
in the  equations   (\ref{eq26}),  we obtain equations  
$$F_3(x_{12}) = \sum_{\ell=0}^{16} a_\ell^{} {x_{12}}^{\ell}_{} =0$$
where $a_\ell$ are given as follows:
   \begin{eqnarray*} 
&&   a_0=k {k_1}^7 \left( k^2(p-2)+{k_1} k+1\right)
   \left(k^2 (p-2) (p-1)+k {k_1}
   (p-1)+ {k_1}^2+p-2\right),
   \\
   &&
 a_1=  -2 k^2
   {k_1}^7 (p k-k+{k_1})
   \left(2 p^2 k^2-6 p k^2+4 k^2-2
   {k_1} k+2 {k_1} p
   k+{k_1}^2+2 p-3\right),
 \\   && a_2 = k^2
   {k_1}^6 \big(6 p^4 k^4-34
   p^3 k^4+70 p^2 k^4-62 p k^4+20
   k^4+18 {k_1} p^3 k^3-60
   {k_1} p^2 k^3  -18 {k_1} k^3
  \\
   && +60 {k_1} p k^3+16 p^3
   k^2+17 {k_1}^2 k^2+23
   {k_1}^2 p^2 k^2-67 p^2 k^2-43
   {k_1}^2 p k^2+87 p k^2 -33 k^2-6 {k_1}^3 k
    \\
   &&+18 {k_1}
   p^2 k+12 {k_1} k+12
   {k_1}^3 p k-36 {k_1} p
   k+{k_1}^4-7 {k_1}^2+10
   p^2+9 {k_1}^2 p-27
   p+14\big), \\
  && 
  a_ 3 = -2 k {k_1}^6 (p
   k-k+{k_1}) \big(12 p^3
   k^4-44 p^2 k^4+50 p k^4-18 k^4+14
   {k_1} p^2 k^3+6 {k_1}
   k^3 \\
   &&
   -20 {k_1} p k^3-3
   {k_1}^2 k^2+20 p^2 k^2+7
   {k_1}^2 p k^2-47 p k^2+23
   k^2-4 {k_1} k+4 {k_1} p
   k+2 {k_1}^2+4 p-6\big), \\
   &&
  a_4= k {k_1}^5 \big(15 p^5 k^6-95
   p^4 k^6+236 p^3 k^6-288 p^2
   k^6+173 p k^6-41 k^6+66 {k_1}
   p^4 k^5\\
   &&
-250 {k_1} p^3 k^5+324
   {k_1} p^2 k^5+22 {k_1}
   k^5-162 {k_1} p k^5+50 p^4
   k^4+97 {k_1}^2 p^3 k^4-245
   p^3 k^4\\
   &&
-30 {k_1}^2 k^4-230
   {k_1}^2 p^2 k^4+436 p^2
   k^4+163 {k_1}^2 p k^4-336 p
   k^4+95 k^4+11 {k_1}^3 k^3+82
   {k_1} p^3 k^3\\
   &&
+52 {k_1}^3
   p^2 k^3-212 {k_1} p^2 k^3-35
   {k_1} k^3-54 {k_1}^3 p
   k^3+156 {k_1} p k^3-4
   {k_1}^4 k^2+39 p^3 k^2+52
   {k_1}^2 k^2\\
   &&
+65 {k_1}^2
   p^2 k^2-142 p^2 k^2+6 {k_1}^4
   p k^2-111 {k_1}^2 p k^2+165 p
   k^2-70 k^2-16 {k_1}^3 k+8
   {k_1} k\\
   &&
+16 {k_1}^3 p k-8
   {k_1} p k+4 {k_1}^4-12
   {k_1}^2-8 p+16\big),
   \\
   &&
  a_5= -2 k^2
   {k_1}^5 (p k-k+{k_1})
   \big(30 p^4 k^4-130 p^3 k^4+202
   p^2 k^4-134 p k^4+32 k^4+42
   {k_1} p^3 k^3\\
   &&
-78 {k_1}
   p^2 k^3+36 {k_1} p k^3+70 p^3
   k^2+21 {k_1}^2 p^2 k^2-221
   p^2 k^2-18 {k_1}^2 p k^2+224
   p k^2-76 k^2\\
   &&
+24 {k_1} p^2
   k+12 {k_1} k-36 {k_1} p
   k-6 {k_1}^2+28 p^2+12
   {k_1}^2 p-72
   p+38\big),
  \\
   &&
a_6=
   {k_1}^4 \big(20
   p^6 k^8-140 p^5 k^8+404 p^4
   k^8-616 p^3 k^8+524 p^2 k^8-236 p
   k^8+44 k^8+130 {k_1} p^5
   k^7\\
   &&
-560 {k_1} p^4 k^7+884
   {k_1} p^3 k^7-592 {k_1}
   p^2 k^7+16 {k_1} k^7+122
   {k_1} p k^7+80 p^5 k^6+211
   {k_1}^2 p^4 k^6\\
   &&
-449 p^4
   k^6-613 {k_1}^2 p^3 k^6+1013
   p^3 k^6+12 {k_1}^2 k^6+605
   {k_1}^2 p^2 k^6-1163 p^2
   k^6-215 {k_1}^2 p k^6\\
   &&
+683 p
   k^6-164 k^6+210 {k_1} p^4
   k^5-8 {k_1}^3 k^5+116
   {k_1}^3 p^3 k^5-688 {k_1}
   p^3 k^5-182 {k_1}^3 p^2
   k^5
\\
   &&
+774 {k_1} p^2 k^5+24
   {k_1} k^5+74 {k_1}^3 p
   k^5-320 {k_1} p k^5+6
   {k_1}^4 k^4+76 p^4 k^4+225
   {k_1}^2 p^3 k^4-359 p^3
   k^4\\
   &&
-112 {k_1}^2 k^4+15
   {k_1}^4 p^2 k^4-577
   {k_1}^2 p^2 k^4+668 p^2
   k^4-20 {k_1}^4 p k^4+491
   {k_1}^2 p k^4-587 p k^4\\
   &&
+174
   k^4+48 {k_1}^3 k^3+16
   {k_1} p^3 k^3+80 {k_1}^3
   p^2 k^3-88 {k_1} p^2 k^3+8
   {k_1} k^3-144 {k_1}^3 p
   k^3+80 {k_1} p k^3\\
   &&
-16 {k_1}^4 k^2+52 {k_1}^2
   k^2+32 {k_1}^2 p^2 k^2-56 p^2
   k^2+20 {k_1}^4 p k^2-116
   {k_1}^2 p k^2+144 p k^2-60
   k^2\\
   &&
-8 {k_1}^3 k-16 {k_1}
   k+16 {k_1}^3 p k+8 {k_1}
   p k+4 {k_1}^4-4
   {k_1}^2\big),
        \end{eqnarray*} 
        \newpage
   \begin{eqnarray*}
   &&
a_7=-2 k
   {k_1}^4 (p k-k+{k_1})
   \big(40 p^5 k^6-200 p^4 k^6+388
   p^3 k^6-364 p^2 k^6+164 p k^6-28
   k^6\\
   &&
+70 {k_1} p^4 k^5-160
   {k_1} p^3 k^5+90 {k_1}
   p^2 k^5-20 {k_1} k^5+20
   {k_1} p k^5+120 p^4 k^4+35
   {k_1}^2 p^3 k^4\\
   && -473 p^3 k^4+10 {k_1}^2 k^4-45
   {k_1}^2 p^2 k^4+711 p^2
   k^4-484 p k^4+126 k^4+60
   {k_1} p^3 k^3-120 {k_1}
   p^2 k^3\\
   &&
+16 {k_1} k^3+44
   {k_1} p k^3+72 p^3 k^2-8
   {k_1}^2 k^2+30 {k_1}^2
   p^2 k^2-274 p^2 k^2-30
   {k_1}^2 p k^2+330 p k^2\\
   &&
-120 k^2-24 p+24\big),
 \\   &&
a_8= k {k_1}^3
   \big(15 p^7 k^8-115 p^6 k^8+376
   p^5 k^8-680 p^4 k^8+735 p^3
   k^8-475 p^2 k^8+170 p k^8-26
   k^8\\
   &&
+150 {k_1} p^6 k^7-730
   {k_1} p^5 k^7+1356 {k_1}
   p^4 k^7-1124 {k_1} p^3
   k^7+286 {k_1} p^2 k^7-64
   {k_1} k^7+126 {k_1} p
   k^7
   \\
    &&
+70 p^6 k^6+265 {k_1}^2
   p^5 k^6-441 p^5 k^6-910
   {k_1}^2 p^4 k^6+1198 p^4
   k^6+1127 {k_1}^2 p^3 k^6-1812
   p^3 k^6\\
   &&
+13 {k_1}^2 k^6-571
   {k_1}^2 p^2 k^6+1608 p^2
   k^6+76 {k_1}^2 p k^6-787 p
   k^6+164 k^6+310 {k_1} p^5
   k^5  \\
   &&
+150 {k_1}^3 p^4 k^5-1246
   {k_1} p^4 k^5+11 {k_1}^3
   k^5-310 {k_1}^3 p^3 k^5+1870
   {k_1} p^3 k^5+181 {k_1}^3
   p^2 k^5
  \\  &&
-1135 {k_1} p^2
   k^5+107 {k_1} k^5-32
   {k_1}^3 p k^5+94 {k_1} p
   k^5+79 p^5 k^4-4 {k_1}^4
   k^4+395 {k_1}^2 p^4 k^4\\
   &&
-462
   p^4 k^4+20 {k_1}^4 p^3
   k^4-1337 {k_1}^2 p^3 k^4+1182
   p^3 k^4+69 {k_1}^2 k^4-40
   {k_1}^4 p^2 k^4+1619
   {k_1}^2 p^2 k^4 \\
   &&
 -1596 p^2
   k^4+24 {k_1}^4 p k^4-746
   {k_1}^2 p k^4+1051 p k^4-254
   k^4+64 {k_1} p^4 k^3-64
   {k_1}^3 k^3\\
   &&
+160 {k_1}^3 p^3 k^3-368 {k_1} p^3 k^3-416
   {k_1}^3 p^2 k^3+480 {k_1}
   p^2 k^3-64 {k_1} k^3+272
   {k_1}^3 p k^3-64 {k_1} p
   k^3\\
   &&
+24 {k_1}^4 k^2+128
   {k_1}^2 p^3 k^2-136 p^3
   k^2-160 {k_1}^2 k^2+40
   {k_1}^4 p^2 k^2-528
   {k_1}^2 p^2 k^2+456 p^2
   k^2\\
   &&
-64 {k_1}^4 p k^2+544
   {k_1}^2 p k^2-448 p k^2+144
   k^2+80 {k_1}^3 k+64
   {k_1}^3 p^2 k+48 {k_1}
   p^2 k+48 {k_1} k\\
   &&
-120
   {k_1}^3 p k-120 {k_1} p
   k-24 {k_1}^4+32
   {k_1}^2+16 {k_1}^4 p-8
   {k_1}^2 p+16 p-32\big),
  \\
   &&
a_9=  -2
   {k_1}^3 (p k-k+{k_1})
   \big(30 p^6 k^8-170 p^5 k^8+392
   p^4 k^8-468 p^3 k^8+302 p^2
   k^8-98 p k^8\\
   &&
+12 k^8+70 {k_1}
   p^5 k^7-190 {k_1} p^4 k^7+120
   {k_1} p^3 k^7+80 {k_1}
   p^2 k^7+30 {k_1} k^7-110
   {k_1} p k^7+110 p^5 k^6\\
   &&
+35
   {k_1}^2 p^4 k^6-517 p^4
   k^6-60 {k_1}^2 p^3 k^6+1016
   p^3 k^6-15 {k_1}^2 k^6-1028
   p^2 k^6+40 {k_1}^2 p k^6\\
   &&
+526
   p k^6-107 k^6+80 {k_1} p^4
   k^5-200 {k_1} p^3 k^5+56
   {k_1} p^2 k^5-88 {k_1}
   k^5+152 {k_1} p k^5+88 p^4
   k^4\\
   &&
+40 {k_1}^2 p^3 k^4-448
   p^3 k^4+44 {k_1}^2 k^4-60
   {k_1}^2 p^2 k^4+836 p^2
   k^4-32 {k_1}^2 p k^4-664 p
   k^4+196 k^4\\
   &&
+48 {k_1} k^3-48
   {k_1} p k^3-24 {k_1}^2
   k^2-104 p^2 k^2+208 p k^2-80
   k^2-16 {k_1} k+16 {k_1} p
   k+8
   {k_1}^2-8\big),
   \\
   &&
a_{10}= {k_1}^2
   \big(6 p^8 k^{10}-50 p^7
   k^{10}+182 p^6 k^{10}-378 p^5
   k^{10}+490 p^4 k^{10}-406 p^3
   k^{10}+210 p^2 k^{10}\\
   &&
-62 p
   k^{10}+8 k^{10}+102 {k_1} p^7
   k^9-556 {k_1} p^6 k^9+1188
   {k_1} p^5 k^9-1170 {k_1}
   p^4 k^9+350 {k_1} p^3 k^9\\
   &&
+288
   {k_1} p^2 k^9+62 {k_1}
   k^9-264 {k_1} p k^9+32 p^7
   k^8+197 {k_1}^2 p^6 k^8-221
   p^6 k^8-773 {k_1}^2 p^5
   k^8\\
   &&
+697 p^5 k^8+1123 {k_1}^2
   p^4 k^8-1307 p^4 k^8-702
   {k_1}^2 p^3 k^8+1558 p^3
   k^8+{k_1}^2 k^8+143
   {k_1}^2 p^2 k^8 \\
   &&
-1159 p^2
   k^8+11 {k_1}^2 p k^8+489 p
   k^8-89 k^8+262 {k_1} p^6
   k^7+116 {k_1}^3 p^5 k^7-1252
   {k_1} p^5 k^7\\
   &&
-290 {k_1}^3
   p^4 k^7+2344 {k_1} p^4 k^7-30
   {k_1}^3 k^7+204 {k_1}^3
   p^3 k^7-1904 {k_1} p^3 k^7-32
   {k_1}^3 p^2 k^7
 \\
   &&
+238 {k_1}
   p^2 k^7-236 {k_1} k^7+32
   {k_1}^3 p k^7+548 {k_1} p
   k^7+42 p^6 k^6+371 {k_1}^2
   p^5 k^6-289 p^5 k^6
\\
   &&
+{k_1}^4
   k^6+15 {k_1}^4 p^4 k^6-1553
   {k_1}^2 p^4 k^6+938 p^4
   k^6-40 {k_1}^4 p^3 k^6+2414
   {k_1}^2 p^3 k^6-1734 p^3
   k^6\\
   &&
+{k_1}^2 k^6+36
   {k_1}^4 p^2 k^6-1652
   {k_1}^2 p^2 k^6+1826 p^2
   k^6-12 {k_1}^4 p k^6+419
   {k_1}^2 p k^6-1009 p k^6 
          \end{eqnarray*} 
   \begin{eqnarray*}
   &&
+226
   k^6+96 {k_1} p^5 k^5+160
   {k_1}^3 p^4 k^5-656 {k_1}
   p^4 k^5+96 {k_1}^3 k^5-544
   {k_1}^3 p^3 k^5+1200
   {k_1} p^3 k^5 
   \\
    &&
+512 {k_1}^3
   p^2 k^5-592 {k_1} p^2 k^5+192
   {k_1} k^5-224 {k_1}^3 p
   k^5-240 {k_1} p k^5-16
   {k_1}^4 k^4+192 {k_1}^2
   p^4 k^4 
   \\
     &&
-136 p^4 k^4+40
   {k_1}^4 p^3 k^4-1000
   {k_1}^2 p^3 k^4+576 p^3
   k^4+176 {k_1}^2 k^4-96
   {k_1}^4 p^2 k^4+1616
   {k_1}^2 p^2 k^4\\
   &&
-960 p^2
   k^4+72 {k_1}^4 p k^4-1064
   {k_1}^2 p k^4+784 p k^4-184
   k^4-176 {k_1}^3 k^3+96
   {k_1}^3 p^3 k^3\\
   &&
+32 {k_1}
   p^3 k^3-304 {k_1}^3 p^2
   k^3-112 {k_1} p^2 k^3-96
   {k_1} k^3+384 {k_1}^3 p
   k^3+176 {k_1} p k^3+40
   {k_1}^4 k^2\\
   &&
-120 {k_1}^2
   k^2+24 {k_1}^4 p^2 k^2-128
   {k_1}^2 p^2 k^2+80 p^2 k^2-72
   {k_1}^4 p k^2+336 {k_1}^2
   p k^2-208 p k^2+48 k^2\\
   &&
+64
   {k_1}^3 k+32 {k_1} k-64
   {k_1}^3 p k-32 {k_1} p
   k-16 {k_1}^4+16
   {k_1}^2\big), \\
       &&
a_{11} = -2 k
   {k_1}^2 (p-1) (p
   k-k+{k_1}) \big(12 p^6
   k^8-64 p^5 k^8+138 p^4 k^8-152
   p^3 k^8+88 p^2 k^8  \\
 &&
-24 p k^8 +2 k^8 +42 {k_1} p^5 k^7-90
   {k_1} p^4 k^7+120 {k_1}
   p^2 k^7+18 {k_1} k^7-90
   {k_1} p k^7+52 p^5 k^6 \\
 &&
+21
   {k_1}^2 p^4 k^6-229 p^4
   k^6-24 {k_1}^2 p^3 k^6+448
   p^3 k^6-9 {k_1}^2 k^6-24
   {k_1}^2 p^2 k^6-460 p^2
   k^6 \\
 &&
+36 {k_1}^2 p k^6+232 p
   k^6-43 k^6+60 {k_1} p^4
   k^5-120 {k_1} p^3 k^5-96
   {k_1} p^2 k^5-108 {k_1}
   k^5+264 {k_1} p k^5 
   \\
 &&
+52 p^4
   k^4+30 {k_1}^2 p^3 k^4-278
   p^3 k^4+54 {k_1}^2 k^4-30
   {k_1}^2 p^2 k^4+550 p^2
   k^4-78 {k_1}^2 p k^4-474 p
   k^4 \\
 &&
+158 k^4+144 {k_1} k^3-144
   {k_1} p k^3-72 {k_1}^2
   k^2-136 p^2 k^2+272 p k^2-144
   k^2-48 {k_1} k+48 {k_1} p
   k \\
 &&
+24 {k_1}^2+24\big),
   \\
   &&
a_{12}={k_1}
   \big(p^9 k^{11}-9 p^8 k^{11}+36
   p^7 k^{11}-84 p^6 k^{11}+126 p^5
   k^{11}-126 p^4 k^{11}+84 p^3
   k^{11}-36 p^2 k^{11} \\
 &&
+9 p
   k^{11}-k^{11}+38 {k_1} p^8
   k^{10}-230 {k_1} p^7
   k^{10}+556 {k_1} p^6
   k^{10}-634 {k_1} p^5
   k^{10}+220 {k_1} p^4
   k^{10} \\
 &&
+262 {k_1} p^3
   k^{10}-340 {k_1} p^2
   k^{10}-26 {k_1} k^{10}+154
   {k_1} p k^{10}+6 p^8 k^9+83
   {k_1}^2 p^7 k^9-43 p^7
   k^9 \\
 &&
-358 {k_1}^2 p^6 k^9+152
   p^6 k^9+569 {k_1}^2 p^5
   k^9-348 p^5 k^9-380 {k_1}^2
   p^4 k^9+545 p^4 k^9+85
   {k_1}^2 p^3 k^9 \\
 &&
-571 p^3
   k^9-24 {k_1}^2 k^9-38
   {k_1}^2 p^2 k^9+378 p^2
   k^9+63 {k_1}^2 p k^9-142 p
   k^9+23 k^9+118 {k_1} p^7
   k^8 \\
 &&
+52 {k_1}^3 p^6 k^8-652
   {k_1} p^6 k^8-146 {k_1}^3
   p^5 k^8+1464 {k_1} p^5
   k^8+101 {k_1}^3 p^4 k^8-1513
   {k_1} p^4 k^8 \\
 &&
+37 {k_1}^3
   k^8+16 {k_1}^3 p^3 k^8+322
   {k_1} p^3 k^8+26 {k_1}^3
   p^2 k^8+774 {k_1} p^2 k^8+183
   {k_1} k^8-86 {k_1}^3 p
   k^8 \\
 &&
-696 {k_1} p k^8+9 p^7
   k^7+179 {k_1}^2 p^6 k^7-66
   p^6 k^7+6 {k_1}^4 p^5 k^7-889
   {k_1}^2 p^5 k^7+261 p^5
   k^7 \\
 &&
-20 {k_1}^4 p^4 k^7+1658
   {k_1}^2 p^4 k^7-654 p^4
   k^7+24 {k_1}^4 p^3 k^7-1472
   {k_1}^2 p^3 k^7+1027 p^3
   k^7+42 {k_1}^2 k^7 \\
 &&
-12 {k_1}^4 p^2 k^7+673
   {k_1}^2 p^2 k^7-966 p^2 k^7+2
   {k_1}^4 p k^7-191 {k_1}^2
   p k^7+495 p k^7-106 k^7 \\
 &&
+64 {k_1} p^6 k^6+80 {k_1}^3
   p^5 k^6-520 {k_1} p^5 k^6-336
   {k_1}^3 p^4 k^6+1240
   {k_1} p^4 k^6-144 {k_1}^3
   k^6 \\
 &&
+384 {k_1}^3 p^3 k^6-1072
   {k_1} p^3 k^6-224 {k_1}^3
   p^2 k^6-112 {k_1} p^2 k^6-296
   {k_1} k^6+240 {k_1}^3 p
   k^6 \\
 &&
+696 {k_1} p k^6+128
   {k_1}^2 p^5 k^5-48 p^5 k^5+4
   {k_1}^4 k^5+20 {k_1}^4
   p^4 k^5-812 {k_1}^2 p^4
   k^5+248 p^4 k^5 \\
 &&
-64 {k_1}^4
   p^3 k^5+1784 {k_1}^2 p^3
   k^5-632 p^3 k^5-68 {k_1}^2
   k^5+72 {k_1}^4 p^2 k^5-1856
   {k_1}^2 p^2 k^5+936 p^2
   k^5 \\
 &&
-32 {k_1}^4 p k^5+824
   {k_1}^2 p k^5-632 p k^5+128
   k^5+64 {k_1}^3 p^4 k^4-48
   {k_1} p^4 k^4+224 {k_1}^3
   k^4 \\
 &&
-272 {k_1}^3 p^3 k^4+208
   {k_1} p^3 k^4+512 {k_1}^3
   p^2 k^4-80 {k_1} p^2 k^4+128
   {k_1} k^4-464 {k_1}^3 p k^4 -272 {k_1} p k^4 \\
 &&
-24 {k_1}^4 k^3+16 {k_1}^4
   p^3 k^3-248 {k_1}^2 p^3
   k^3+96 p^3 k^3+176 {k_1}^2
   k^3-72 {k_1}^4 p^2 k^3+840
   {k_1}^2 p^2 k^3 \\
 &&
-336 p^2
   k^3+80 {k_1}^4 p k^3-768
   {k_1}^2 p k^3+288 p k^3-48
   k^3-192 {k_1}^3 k^2-128
   {k_1}^3 p^2 k^2-96 {k_1}
   p^2 k^2 \\
 &&
-32 {k_1} k^2+256
   {k_1}^3 p k^2+192 {k_1} p
   k^2+32 {k_1}^4 k-32
   {k_1}^2 k-32 {k_1}^4 p
   k+32 {k_1}^2 p k-16 p k \\
 &&
+16 k+16 {k_1}^3-16 {k_1}\big),
     \end{eqnarray*} 
   \begin{eqnarray*}
 &&
a_{13} = -2 {k_1} (p k-k+{k_1}) \big(2 p^8
   k^{10}-14 p^7 k^{10}+42 p^6
   k^{10}-70 p^5 k^{10}+70 p^4
   k^{10}-42 p^3 k^{10} \\
 &&
+14 p^2
   k^{10}-2 p k^{10}+14 {k_1}
   p^7 k^9-50 {k_1} p^6 k^9+36
   {k_1} p^5 k^9+80 {k_1}
   p^4 k^9-170 {k_1} p^3 k^9 \\
 &&
+126
   {k_1} p^2 k^9+4 {k_1}
   k^9-40 {k_1} p k^9+10 p^7
   k^8+7 {k_1}^2 p^6 k^8-59 p^6
   k^8-18 {k_1}^2 p^5 k^8+168
   p^5 k^8 \\
 &&
-288 p^4 k^8+40
   {k_1}^2 p^3 k^8+302 p^3 k^8-2
   {k_1}^2 k^8-45 {k_1}^2
   p^2 k^8-183 p^2 k^8+18
   {k_1}^2 p k^8+56 p k^8 \\
 &&
-6 k^8+24 {k_1} p^6 k^7-84
   {k_1} p^5 k^7-4 {k_1} p^4
   k^7+328 {k_1} p^3 k^7-480
   {k_1} p^2 k^7-52 {k_1}
   k^7+268 {k_1} p k^7 \\
 &&
+12 p^6
   k^6+12 {k_1}^2 p^5 k^6-88 p^5
   k^6-30 {k_1}^2 p^4 k^6+278
   p^4 k^6-32 {k_1}^2 p^3
   k^6-488 p^3 k^6+26 {k_1}^2
   k^6 \\
 &&
+132 {k_1}^2 p^2 k^6+488
   p^2 k^6-108 {k_1}^2 p k^6-256
   p k^6+54 k^6-144 {k_1} p^3
   k^5+432 {k_1} p^2 k^5 \\
 &&
+144
   {k_1} k^5-432 {k_1} p
   k^5-56 p^4 k^4+224 p^3 k^4-72
   {k_1}^2 k^4-72 {k_1}^2
   p^2 k^4-392 p^2 k^4+144
   {k_1}^2 p k^4 \\
 &&
+336 p k^4-112
   k^4+48 {k_1} p^3 k^3-144
   {k_1} p^2 k^3-112 {k_1}
   k^3+208 {k_1} p k^3+56
   {k_1}^2 k^2+24 {k_1}^2
   p^2 k^2 \\
 &&
+72 p^2 k^2-48 {k_1}^2
   p k^2-144 p k^2+40 k^2+32
   {k_1} k-32 {k_1} p k-16
   {k_1}^2+16\big),
 \\
   &&
a_{14}=6 {k_1}
   p^9 k^{11}-40 {k_1} p^8
   k^{11}+108 {k_1} p^7
   k^{11}-140 {k_1} p^6
   k^{11}+56 {k_1} p^5 k^{11}+84
   {k_1} p^4 k^{11}   \\
   &&
-140
   {k_1} p^3 k^{11}+92 {k_1}
   p^2 k^{11}+4 {k_1} k^{11}-30
   {k_1} p k^{11}+17 {k_1}^2
   p^8 k^{10}+p^8 k^{10}-75
   {k_1}^2 p^7 k^{10}   \\
   &&
-5 p^7
   k^{10}+111 {k_1}^2 p^6
   k^{10}+7 p^6 k^{10}-43
   {k_1}^2 p^5 k^{10}+7 p^5
   k^{10}-15 {k_1}^2 p^4
   k^{10}-35 p^4 k^{10}   \\
   &&
-57
   {k_1}^2 p^3 k^{10}+49 p^3
   k^{10}+18 {k_1}^2 k^{10}+125
   {k_1}^2 p^2 k^{10}-35 p^2
   k^{10}-81 {k_1}^2 p k^{10}+13
   p k^{10}-2 k^{10}   \\
   &&
+22 {k_1}
   p^8 k^9+12 {k_1}^3 p^7
   k^9-136 {k_1} p^7 k^9-34
   {k_1}^3 p^6 k^9+358 {k_1}
   p^6 k^9+10 {k_1}^3 p^5
   k^9-464 {k_1} p^5 k^9   \\
   &&
+40
   {k_1}^3 p^4 k^9+170 {k_1}
   p^4 k^9-20 {k_1}^3 k^9+328
   {k_1} p^3 k^9-82 {k_1}^3
   p^2 k^9-494 {k_1} p^2 k^9-56
   {k_1} k^9   \\
   &&
+74 {k_1}^3 p
   k^9+272 {k_1} p k^9+35
   {k_1}^2 p^7 k^8+3 p^7
   k^8+{k_1}^4 p^6 k^8-199
   {k_1}^2 p^6 k^8-12 p^6 k^8-4
   {k_1}^4 p^5 k^8   \\
   &&
+407
   {k_1}^2 p^5 k^8+p^5 k^8+6
   {k_1}^4 p^4 k^8-412
   {k_1}^2 p^4 k^8+74 p^4 k^8-4
   {k_1}^4 p^3 k^8+333
   {k_1}^2 p^3 k^8   \\
   &&
-171 p^3
   k^8-74 {k_1}^2
   k^8+{k_1}^4 p^2 k^8-347
   {k_1}^2 p^2 k^8+176 p^2
   k^8+257 {k_1}^2 p k^8-89 p
   k^8+18 k^8   \\
   &&
+16 {k_1} p^7
   k^7+16 {k_1}^3 p^6 k^7-152
   {k_1} p^6 k^7-80 {k_1}^3
   p^5 k^7+448 {k_1} p^5 k^7+80
   {k_1}^3 p^4 k^7-568 {k_1}
   p^4 k^7   \\
   &&
+112 {k_1}^3 k^7-32
   {k_1}^3 p^3 k^7+112 {k_1}
   p^3 k^7+176 {k_1}^3 p^2
   k^7+536 {k_1} p^2 k^7+184
   {k_1} k^7-272 {k_1}^3 p
   k^7   \\
   &&
-576 {k_1} p k^7+32
   {k_1}^2 p^6 k^6+4 {k_1}^4
   p^5 k^6-228 {k_1}^2 p^5
   k^6-16 {k_1}^4 p^4 k^6+636
   {k_1}^2 p^4 k^6-44 p^4 k^6   \\
   &&
+24
   {k_1}^4 p^3 k^6-968
   {k_1}^2 p^3 k^6+192 p^3
   k^6+60 {k_1}^2 k^6-16
   {k_1}^4 p^2 k^6+840
   {k_1}^2 p^2 k^6-296 p^2 k^6   \\
   &&
+4
   {k_1}^4 p k^6-372 {k_1}^2
   p k^6+192 p k^6-44 k^6+16
   {k_1}^3 p^5 k^5-40 {k_1}
   p^5 k^5-72 {k_1}^3 p^4
   k^5   \\
   &&
+224 {k_1} p^4 k^5-200
   {k_1}^3 k^5+192 {k_1}^3
   p^3 k^5-288 {k_1} p^3 k^5-304
   {k_1}^3 p^2 k^5-80 {k_1}
   p^2 k^5-144 {k_1} k^5   \\
   &&
+368
   {k_1}^3 p k^5+328 {k_1} p
   k^5+4 {k_1}^4 k^4+4
   {k_1}^4 p^4 k^4-124
   {k_1}^2 p^4 k^4+16 p^4 k^4-24
   {k_1}^4 p^3 k^4   \\
   &&
+544
   {k_1}^2 p^3 k^4-80 p^3 k^4-68
   {k_1}^2 k^4+40 {k_1}^4
   p^2 k^4-848 {k_1}^2 p^2
   k^4+192 p^2 k^4-24 {k_1}^4 p
   k^4   \\
   &&
+560 {k_1}^2 p k^4-208 p
   k^4+16 k^4+192 {k_1}^3 k^3-64
   {k_1}^3 p^3 k^3+192
   {k_1}^3 p^2 k^3-320
   {k_1}^3 p k^3   \\
   &&
-16 {k_1}^4
   k^2+96 {k_1}^2 k^2-16
   {k_1}^4 p^2 k^2+160
   {k_1}^2 p^2 k^2-48 p^2 k^2+32
   {k_1}^4 p k^2-320 {k_1}^2
   p k^2   \\
   &&
+96 p k^2+16 k^2-96
   {k_1}^3 k+96 {k_1}^3 p
   k+16 {k_1}^4-16  {k_1}^2,
   \\
   &&
a_{15}=
   -2 k (p-1) (p
   k-k+{k_1}) \left(p
   k^2-k^2+2\right)
   \left({k_1}^2-2 k {k_1}+2
   k p {k_1}+1\right)\times  \\
   &&
 \left(p^2
   k^2-p k^2-2\right) \left(p^3
   k^4-3 p k^4+2 k^4-8
   k^2+4\right),
   \\
   &&
a_{16}= k (p-1) ((p-1)k+{k_1}) ((p-1)k {k_1}+1) 
\left(k^4 (p-1)^2 (p+2)-8 k^2+4\right)^2. 
    \end{eqnarray*} 
We see that the value of $F_3(x_{12})$ at $x_{12} = 1$ is given by 
    \begin{eqnarray*} 
 && F_3(1)= (k-1) (k+1) (k-{k_ 1}) (k( p-1)+{k_ 1}-2)^2 (k( p-1)+{k_ 1}+2)^2 \times
  \\
  && \Big({k_ 1}^3 \left(3 k^2
   (p-1)+4\right)+3 k {k_ 1}^2 \left(k^2
   (p-1)^2+3 p-2\right)+k \left(k^2 (p-1)^2
   (p+2)-4\right)^3\\
   &&+{k_ 1} \left(k^4
   (p-1)^3+6 k^2 (p-1) p-4\right)+k
   {k_ 1}^4\Big). 
       \end{eqnarray*} 
Hence it follows that, for $k_1 > k$ it is  $ F_3(1) < 0$.  
We also see that $ F_3(0) =a_0 > 0 $ and the leading coefficient $a_{16}$ of 
$ F_3(x_{12}) $ is positive. In fact, we see that $ F_3(k_1 k p) >0$. 

 In conclusion, we see that for $k_1 > k$,    $ F_3(x_{12})=0 $ has at least two positive solutions
 $x_{12}= \alpha_1, \alpha_2$, with $ 0 < \alpha_1 <1 < \alpha_2 < k_1 k p $. 

\smallskip
\noindent
{\bf Remark 1. } It is possible to show that, for $k_1 \geq 8 k p$, $k \geq 2$ and $ p \geq 3$,  the equation $F_3(x_{12}) = 0$ has at least four positive solutions.  We just give a  sketch of proof. 

1) \ We see that $F_3(2) >0$ for $k_1 \geq 8 k p$, $k \geq 2$ and $ p \geq 3$. 

2) \ Put $\displaystyle \beta =\frac{ \left(k^2 (p-1)+2\right) \left(k^2 (p-1) p-2\right)}{k (p-1) \left(k^4 (p-1)^2 (p+2)-8
   k^2+4\right)} k_1$. We see that, for  $k_1 \geq 5 k p$, $k \geq 2$ and $ p \geq 3$, $F_3(\beta) < 0$. 
   
 3) \ We have $\beta > 2$ for  $k_1 \geq 5 k p$, $k \geq 2$ and $ p \geq 3$. 
 
 Thus  there are at least four positive solutions $\alpha_1, \alpha_2, \alpha_3, \alpha_4$ of the equation $F_3(x_{12}) = 0$ with $$ 0 < \alpha_1 <1 < \alpha_2 <  2 < \alpha_3 < \beta  < \alpha_4  < k_1 k p, $$ 
 for $k_1 \geq 8 k p$, $k \geq 2$ and $ p \geq 3$. 
 Therefore, it follows that the Lie group  $\SU(N)$  ($N=k_1 +(p-1)k$) admits at least four  $\Ad(\s(\U(k_1)\times \U(k)\times \cdots \times \U(k))$-invariant Einstein metrics, which are not naturally reductive. 

 We now examine the case   
for $k_1 = k$.   We see that  
$$ F_3(x_{12}) = k^2 (x_{12}-1) G_3(x_{12}),$$
where $ \displaystyle G_3(x_{12}) = \sum_{\ell=0}^{15} b_{\ell}^{} {x_{12}}^{\ell}_{}$ is 
a  polynomial of $x_{12}$ with coefficients 
   \begin{eqnarray*} 
&&   b_0 = - k^6 \left( (p-1) k^2+1\right) \left( (p^2 -2 p +2) k^2+p-2\right), \\
   &&  b_1=  k^6 \left(3 p^3 k^4-5 p^2 k^4+2 p k^4+2 k^4+2 p^2 k^2-p
   k^2-4 k^2-p+2\right),\\
  &&  b_2=  -k^6 (6 p^4 k^4-19 p^3 k^4+38 p^2 k^4-35 p
   k^4+12 k^4+16 p^3 k^2-51 p^2 k^2 \\  &&+61 p k^2-24 k^2
   +10 p^2-26  p+12),
   \\
 &&  b_3= k^6 (18 p^4 k^4-41 p^3 k^4+36 p^2 k^4+5 p k^4-12
   k^4+24 p^3 k^2  -35 p^2 k^2  \\  &&-19 p k^2+24 k^2-2 p^2+14 p-12),  
 \\  
 &&  b_4=  -k^4 (15 p^5 k^6-47 p^4 k^6+124 p^3 k^6-178 p^2 k^6+121 p k^6-30
   k^6 +50 p^4 k^4-187 p^3 k^4   \quad  \\  
   && +324 p^2 k^4-256 p k^4+76 k^4+39 p^3 k^2-140
   p^2 k^2+143 p k^2-62 k^2-8 p+16 )
       \end{eqnarray*} 
 \begin{eqnarray*}
     && b_5=  k^4 (45 p^5 k^6-129 p^4
   k^6+166 p^3 k^6-54 p^2 k^6-57 p k^6+30 k^6+90 p^4 k^4-207 p^3 k^4
   \\&&+76
   p^2 k^4+116 p k^4-76 k^4+17 p^3 k^2-4 p^2 k^2-67 p k^2+62 k^2+8
   p-16),
    \\
   && b_6=  -k^4(20 p^6 k^6-55 p^5 k^6+184 p^4 k^6-395 p^3
   k^6+424 p^2 k^6-218 p k^6+40 k^6+80 p^5 k^4 \\
   && -329 p^4 k^4+757 p^3
   k^4-962 p^2 k^4+614 p k^4-144 k^4+76 p^4 k^2-360 p^3 k^2+616 p^2
   k^2 \\
   && -540 p k^2+168 k^2-56 p^2+144 p-64),
   \\  
    && b_7= k^4 (60 p^6 k^6-205
   p^5 k^6+342 p^4 k^6-243 p^3 k^6-56 p^2 k^6+142 p k^6-40 k^6 +160 p^5 k^4 \\
   &&-497 p^4 k^4+485 p^3 k^4+22 p^2 k^4-346 p k^4+144 k^4+68 p^4
   k^2 -188 p^3 k^2+44 p^2 k^2+300 p k^2 \\
   &&-168 k^2+8 p^2-96 p+64 ),
    \\
   && b_8=  -k^2(15 p^7 k^8-25 p^6 k^8+116 p^5 k^8-426 p^4 k^8+691 p^3 k^8-563
   p^2 k^8+222 p k^8-30 k^8 \\
   &&+70 p^6 k^6-291 p^5 k^6+844 p^4 k^6-1604 p^3
   k^6+1694 p^2 k^6-885 p k^6+156 k^6+79 p^5 k^4 \\
   &&-466 p^4 k^4+1130 p^3
   k^4-1624 p^2 k^4+1127 p k^4-254 k^4-136 p^3 k^2+496 p^2 k^2-480 p
   k^2 \\
   &&+160 k^2+16 p-32),
   \\
   && b_9= k^2 (45 p^7 k^8-175 p^6 k^8+358 p^5
   k^8-390 p^4 k^8+73 p^3 k^8+227 p^2 k^8-168 p k^8+30 k^8 \\
   &&+150 p^6
   k^6-583 p^5 k^6+868 p^4 k^6-460 p^3 k^6-402 p^2 k^6+583 p k^6-156
   k^6+97 p^5 k^4 \\
   &&-430 p^4 k^4+542 p^3 k^4+200 p^2 k^4-687 p k^4+254
   k^4-72 p^3 k^2-48 p^2 k^2+304 p k^2 \\
   &&-160 k^2-32 p+32),
    \\
   && b_{10} = -k^2(6 p^8 k^8+7 p^7 k^8-2 p^6 k^8-205 p^5 k^8+558 p^4 k^8-667 p^3
   k^8+418 p^2 k^8-127 p k^8 \\
   &&+12 k^8+32 p^7 k^6-109 p^6 k^6+399 p^5
   k^6-1224 p^4 k^6+2024 p^3 k^6-1755 p^2 k^6+721 p k^6 \\
   &&-88 k^6+42 p^6
   k^4-290 p^5 k^4+904 p^4 k^4-1980 p^3 k^4+2370 p^2 k^4-1314 p k^4+204
   k^4 \\
   &&-136 p^4 k^2+680 p^3 k^2-1152 p^2 k^2+928 p k^2-192 k^2+80 p^2-208
   p+64),\\
   && b_{11} =  k^2 (18 p^8 k^8-75 p^7 k^8+184 p^6 k^8-285 p^5
   k^8+162 p^4 k^8+143 p^3 k^8-240 p^2 k^8 \\
   &&+105 p k^8-12 k^8+72 p^7
   k^6-333 p^6 k^6+655 p^5 k^6-664 p^4 k^6-16 p^3 k^6+725 p^2 k^6  \\
   && -527 p
   k^6+88 k^6+62 p^6 k^4-370 p^5 k^4+752 p^4 k^4-356 p^3 k^4-674 p^2
   k^4+854 p k^4  \\
   && -204 k^4-136 p^4 k^2+232 p^3 k^2+176 p^2 k^2-592 p
   k^2+192 k^2-32 p^2+160 p-64),
   \\
   && b_{12} =  -p^9 k^{10}-11 p^8 k^{10}+36 p^7
   k^{10}+18 p^6 k^{10}-206 p^5 k^{10}+367 p^4 k^{10}-328 p^3 k^{10} \\
   &&+160 p^2 k^{10}  -37 p k^{10}+2 k^{10}-6 p^8 k^8-3 p^7 k^8-12 p^6 k^8+348 p^5
   k^8-1038 p^4 k^8 \\
   &&+1385 p^3 k^8-948 p^2 k^8+294 p k^8-20 k^8-9 p^7
   k^6+64 p^6 k^6-239 p^5 k^6+914 p^4 k^6 \\
   &&-1839 p^3 k^6+1820 p^2 k^6-777 p
   k^6+66 k^6+48 p^5 k^4-336 p^4 k^4+904 p^3 k^4 \\
   &&-1392 p^2 k^4+856 p
   k^4-80 k^4-96 p^3 k^2+400 p^2 k^2-352 p k^2+32 k^2+16 p, 
   \\
   && b_{13} = 3 p^9 k^{10}-11 p^8 k^{10}+34 p^7 k^{10}-86 p^6 k^{10}+94 p^5 k^{10}+23 p^4
   k^{10}-138 p^3 k^{10} \\
   &&+112 p^2 k^{10}-33 p k^{10}+2 k^{10}+14 p^8
   k^8-73 p^7 k^8+180 p^6 k^8-296 p^5 k^8+158 p^4 k^8 \\
   &&+323 p^3 k^8-516 p^2
   k^8+230 p k^8-20 k^8+15 p^7 k^6-112 p^6 k^6+317 p^5 k^6-350 p^4
   k^6 \\
   &&-143 p^3 k^6+732 p^2 k^6-525 p k^6+66 k^6-64 p^5 k^4+208 p^4
   k^4-120 p^3 k^4-400 p^2 k^4 \\
   &&+520 p k^4-80 k^4+48 p^3 k^2+48 p^2 k^2-240
   p k^2+32 k^2+48 p,
          \end{eqnarray*} 
 \begin{eqnarray*}
   && b_{14} = -((p-1) p (p^3 k^4-3 p k^4+2 k^4-8
   k^2+4) (3 p^4 k^6-9 p^3 k^6+11 p^2 k^6-7 p k^6   \quad  \quad \\
   &&+2 k^6 +9 p^3 k^4-24 p^2 k^4+29 p k^4-14 k^4+4 p^2 k^2-28 p k^2+24
   k^2-12)),
   \\
   && 
   b_{15}= (p-1) p \left(k^2 (p-1)+1\right) \left(k^4 (p-1)^2 (p+2)-8 k^2+4\right)^2. 
   \end{eqnarray*} 
   We see that the value of $G_3(x_{12})$ at $x_{12} = 0$ is given by 
   $$ b_0 = - k^6 \left( (p-1) k^2+1\right) \left( (p^2 -2 p +2) k^2+p-2\right) <0$$
  We also see that $$ G_3(1) =2 (k-1)^2 (k+1)^2 p^2 (k p-2) (k p+2) \left(k^4 p^3-12 k^2 p-8
   k^2-8\right).$$ 
   Now we have $k^4 p^3-12 k^2 p-8 k^2-8 = k^4 (p-3)^3+9 k^4 (p-3)^2+\left(27 k^4-12 k^2\right)
   (p-3)+27 k^4-44 k^2-8$ and $27 k^4-44 k^2-8 =27 (k-2)^4+216 (k-2)^3+604 (k-2)^2+688 (k-2)+248$. 
   Hence, we see that $ G_3(1) > 0$  for $k \geq 2$ and $ p \geq 3$. 
   
   Therefore, we see that, for $k_1 = k$,    $ F_3(x_{12})=0 $ has at least two positive solutions  $x_{12}= 1$ and $ \gamma$, with $ 0 < \gamma <1 $.

\section{Non naturally reductive Einstein metrics on $\SU(N)=\SU(k_1+(p-1) k)$}

Let $G$ be a compact, connected semisimple Lie group, $H$ a closed subgroup of $G$ and let $\fr{g}$ be the Lie algebra of $G$ and $\fr{h}$ the subalgebra corresponding to $H$. 
Let $\fr{n}$ be an orthogonal complement of $\fr{h}$ with respect to  $-B$ (the negative of the Killing form of $\fr{g}$).  Then we have
$$
\fr{g}=\fr{h}\oplus\fr{n}, \quad \Ad(H)\fr{n}\subset\fr{n}.
$$

A Riemannian homogeneous space $(M=G/H, g)$ with reductive complement $\fr{n}$ of $\fr{h}$ in $\fr{g}$ is called {\it naturally reductive} if
$$
\langle [X, Y]_\fr{n}, Z\rangle +\langle Y, [X, Z]_\fr{n}\rangle=0 \quad\mbox{for all}\ 
X, Y, Z\in\fr{n}.
$$ 
Here $\langle\ , \  \rangle$ denotes the inner product on $\fr{n}$ induced from the Riemannian metric $g$.  

In \cite{DZ}  D'Atri and Ziller  investigated naturally reductive metrics among left-invariant metrics on compact Lie groups and gave a characterization of metrics in the case of simple Lie groups.

Let $G$ be a compact connected semisimple Lie group, $L$ a closed subgroup of $G$ and $\fr{l}$ the subalgebra corresponding to $L$. 
Let $\fr{l}=\fr{l}_0\oplus\fr{l}_1\oplus\cdots\oplus\fr{l}_p$ be a decomposition of $\fr{l}$ into ideals, where $\fr{l}_0$ is the center of $\fr{l}$ and $\fr{l}_i$ $(i=1,\dots , p)$ are simple ideals of $\fr{l}$.
Let $A_0|_{\fr{l}_0}$ be an arbitrary metric on $\fr{l}_0$. 


\begin{theorem}\label{DZ} {\rm (\cite[Theorem 1, p. 9 and Theorem 3, p. 14]{DZ})}  Under the notations above a left-invariant metric on $G$ of the form
\begin{equation}\label{natural}
\langle\  ,\  \rangle =x\cdot B|_{\fr{n}}+A_0|_{\fr{l}_0}+u_1\cdot B|_{\fr{l}_1}+\cdots
+ u_p\cdot B|_{\fr{l}_p}, \quad (x, u_1, \dots , u_p >0)
\end{equation}
is naturally reductive with respect to $G\times L$, where $G\times L$ acts on $G$ by
$(g, l)y=g y l^{-1}$.

Moreover, if a left-invariant metric $\langle\ ,\ \rangle$ on a compact simple Lie group
$G$ is naturally reductive, then there is a closed subgroup $L$ of $G$ and the metric
$\langle\ ,\ \rangle$ is given by the form $(\ref{natural})$.
\end{theorem}

\noindent
We now consider  the left-invariant metrics $g$  on  $\SU(N)=\SU(k_1+(p-1) k)$ given by (\ref{metric_g111}). 
Let  $\fr{k}=\fr{c}\oplus\fr{su}(k_1)\oplus\underbrace{\fr{su}(k)\oplus\cdots\oplus\fr{su}(k)}_{p-1}=\fr{c}\oplus
\fr{m}_1\oplus\fr{m}_2\oplus\cdots\oplus\fr{m}_{p}$.

\begin{prop}\label{reductive}
If a left invariant metric $g$ of the form   {\em (\ref{metric_g111})} on $\SU(k_1+(p-1) k)$  is naturally reductive with respect to $\SU(k_1+(p-1) k) \times L$, for some closed subgroup $L$ of $\SU(k_1+(p-1) k)$, then  one of 
 the following holds:
 The metric   {\em (\ref{metric_g111})} is either 
\begin{itemize}
\item[(i)]
$\Ad(\s(\U(k_1)\times \underbrace{ \U(k)\times \cdots \times \U(k)}_{p-1}))$-invariant and  $ x_{12}=x_{23}$, or 
\item[(ii)]
 $\Ad(\s(\U(k_1)\times \U((p-1)k)))$-invariant and  $y_2 = x_2=x_{23} $. 
\end{itemize}
Conversely, if one of the conditions {\em (i), (ii)} is satisfied,  then the metric 
 of the form {\em (\ref{metric_g111})}  is  naturally reductive  with respect to $\SU(k_1+(p-1) k)\times L$, for some closed subgroup $L$ of $\SU(k_1+(p-1) k)$.
  \end{prop}
 \begin{proof} 
 We first consider the case when the Lie subalgebra $ \fr{l}$ is not contained in $ \fr{k}$. 
  Let $ \fr{h}$ be the Lie subalgebra generated by $ \fr{l}$ and $ \fr{k}$. 
 Then the invariant metric $g$ is $\ad(\fr{h})$-invariant. 
 Take an element $X \in  \fr{h}$ with $X \notin \fr{k}$. We see that $\displaystyle X \in \fr{m} = \sum_{r < s}\fr{m}_{r s}$ and we write $ \displaystyle X = \sum_{r < s}{X}_{r s}$ where  $ {X}_{r s} \in \fr{m}_{r s}$.  Note that, for $Y_i  \in \fr{m}_{i}$ and  $Y_ j \in \fr{m}_{j}$ ($i < j$), we have $[Y_i, [Y_j, X]] \in   \fr{m}_{i j}$ (cf. Lemma \ref{brackets}).
  Thus,  if the component $X_{i j}$ of $X$ is not zero, we see that the submodule  $\fr{m}_{i j}$ is contained in $\fr{h}$. Hence, the Lie subalgebra $\fr{h}_{i j}$ generated by $\fr{m}_{i}$, $\fr{m}_{j}$ and $\fr{m}_{i j}$, which is isomorphic to $\fr{su}(k_i+ k_j)$, is contained in $\fr{h}$. 
(Note that  $\fr{h}_{i j} =\tilde{\fr{a}}_{i j} +  \fr{m}_{i} +\fr{m}_{j} + \fr{m}_{i j}$, where $\tilde{\fr{a}}_{i j}$ is the center of  $\fr{h}_{i j}$.)
  
 We introduce the new subalgebra $\fr{k}_{i j}$
generated by $\fr{h}_{i j}$ and $\fr{k}$. 
Consider  the $\ad(\fr{k}_{i j})$-module decomposition of $\fr{su}(k_1+(p-1) k)$, Then 
 $\fr{m}_{i s}+\fr{m}_{j s}$ ($s \neq i, j$) is  an irreducible summand. 
Since $\fr{k}_{i j}$ is contained in $\fr{h}$,
the subspace $\fr{m}_{i s}+\fr{m}_{j s}$ is contained in an irreducible
summand  $\fr{n}_{\ell}$ of  the $\ad(\fr{h})$-module decomposition of $\fr{su}(k_1+(p-1) k)$.
Now, for the $\ad(\fr{h})$-invariant metric $g$, we see that
$g_{ |_{\fr{n}_{\ell}}}= z_{\ell} B_{|_{\fr{n}_{\ell}}}$ for some $z_{\ell} > 0$.
Thus we have  $x_{is} = x_{js}$.

  Now, 
  if $X_{1j} \neq  0$  ($j=2, \dots, p$),  then we see that $x_{1s} = x_{js}$ and thus $x_{12} = x_{23}$. 
 If $X_{ij} \neq  0$ ($i\geq 2$), then $x_i = x_{ij}$ and thus $x_{2} = x_{23}$.  
  Moreover,  we have $\fr{h}_{i j}= \fr{k}_{i j}\oplus\fr{m}_{i j}$, where  $\fr{k}_{i j}$ is isomorphic to $\fr{s}(\fr{u}(k_i)\oplus \fr{u}(k_j))$. Let  $\fr{c}_{i j}$ be the center of $\fr{k}_{i j}$. Note that $\fr{c}_{i j}$ is contained in 
$ \sum_{j=2}^{p-1} { \fr{c}_j }$. Thus we see that $y_2 = x_{2}= x_{23}$ and the metric is $\Ad(\s(\U(k_1)\times \U((p-1)k)))$-invariant. 

 Finally, we consider the the case when the Lie subalgebra $ \fr{l}$ is contained in $ \fr{k}$.
 Note that the orthogonal complement of $ \fr{l}$ contains $\fr{m}$. 
 Thus we see that $x_{1r} = x_{js}$ for $r\ge 2$ and $s\ge 3$ and hence, $x_{12} = x_{23}$. 
The converse is a direct consequence of Theorem \ref{DZ}.  
\end{proof}

\smallskip
We now use Proposition \ref{reductive} to determine  which of the Einstein metrics found in Section 4 are naturally reductive or not.  
For {\bf Case  1)}, it follows that these metrics are naturally reductive. 
For
{\bf Case  2)},
the  obtained Einstein metrics 
 satisfy $x_{12} \neq  x_{23}=1$ and $0 <  x_2 < 1 = x_{23}$ for $k_1 > k$.  For $k_1\geq 8 k p$, $k \geq 2$ and $ p \geq 3$, we also have  $x_{12} \neq  x_{23}=1$  and $0 <  x_2 < 1 = x_{23}$.    Thus we see that the Einstein metrics  are  not naturally reductive.  

For the case $k_1=k$, that is, $\SU(p k)$, we see that $F_3(x_{12}) = 0$ has at least two positive solutions, namely  $x_{12}=1$  and one in the interval $(0,1)$. For $x_{12}=1$, we see that  the obtained Einstein metric is naturally reductive  with respect to $\SU(p k)\times\s(\U(k)\times\cdots\times\U(k))$, which is an extension result of   
D'Atri and Ziller \cite[p. 55]{DZ}. 
For the other positive solution $x_{12} < 1$, we also  have $0 <  x_2 < 1 = x_{23}$ and thus  the Einstein metric is  not naturally reductive.   
This concludes the proof of Theorem \ref{main}.

\section{The isometry problem for  metrics on the compact simple Lie group $\SU(N)$}
Recall that we have normalized our metrics by setting $x_{23}=1$, so the Einstein metrics obtained in Section 4 are 
not a homothetic change. 
We consider whether these Einstein metrics  are isometric or not. 

Note that, if two invariant  Einstein metrics $g$ and $g'$ are isometric, then the Einstein constants $\lambda(g)$ and 
$\lambda(g')$ are the same, that is, if $\Ric(g) = \lambda(g) g$ and $\Ric(g') = \lambda(g') g'$,  
then $\lambda(g) = \lambda(g')$. 

\begin{lemma}\label{einstein_constant} 
For the Einstein metrics $g$ obtained in Section 4, the Einstein constants  $ \lambda(g)$ are given by 
 \begin{eqnarray*}
 && \lambda(g) = \left(k (p-1) {x_{12}}^2+{k_{1}}\right) \big(
   \left(k^4 (p-1)^2 (p+2)-8 k^2+4\right){x_{12}}^6 \\
   && + 3 k^3 {k_{1}} (p-1) (p+1) {x_{12}}^4+3 k^2 {k_{1}}^2 p \,
   {x_{12}}^2 +k {k_{1}}^3\big)/ \\
   && \big( 4 (k (p-1)+{k_{1}}){x_{12}}^2 
   \left(\left(k^2 (p-1)+2\right){x_{12}}^2 +k {k_{1}}\right)\times \\ &&
   \left( \left(k^2 (p-1) p-2\right){x_{12}}^4+k {k_{1}} (2 p-1)
   {x_{12}}^2+{k_{1}}^2\right)\big),  
\end{eqnarray*} 
where $x_{12}$ satisfies $F_3(x_{12}) = 0$. 
\end{lemma}
 \begin{proof}  
From  Proposition 3.8 (19), we have $\displaystyle  \lambda(g) = rr_1 = \frac{1}{4}\frac{y_1}{{x_{12}}^2}$. 
From (30) and (34), we see 
 \begin{eqnarray*}
 && y_1 = \left(k (p-1) {x_{12}}^2+{k_{1}}\right) \big(
   \left(k^4 (p-1)^2 (p+2)-8 k^2+4\right){x_{12}}^6 \\
   && + 3 k^3 {k_{1}} (p-1) (p+1) {x_{12}}^4+3 k^2 {k_{1}}^2 p \,
   {x_{12}}^2 +k {k_{1}}^3\big)/ \\
   && \big( (k (p-1)+{k_{1}})
   \left(\left(k^2 (p-1)+2\right){x_{12}}^2 +k {k_{1}}\right)\times \\
   &&
   \left( \left(k^2 (p-1) p-2\right){x_{12}}^4+k {k_{1}} (2 p-1)
   {x_{12}}^2+{k_{1}}^2\right)\big) 
\end{eqnarray*} 
and thus we obtain our lemma. 
\end{proof}
Now we consider $\lambda(g)$ as a function of $x_{12}$. 
\begin{lemma}
For $p \geqq 3$, 
$\lambda(g)$ is strictly monotone decrease for $x_{12} > 0$.
\end{lemma}
 \begin{proof}  
 We see that   
 \begin{eqnarray*} &&
 \frac{d \lambda(g)}{d x_{12}}  = -k_1 Q(x_{12})/ \Big(2(k (p-1)+{k_{1}}) {x_{12}}^3
  \left(\left(k^2 (p-1)+2\right){x_{12}}^2 +k {k_{1}}\right)^2\times
   \\&&
  \left(  \left(k^2 (p-1) p-2\right){x_{12}}^4+k {k_{1}} (2 p-1)
   {x_{12}}^2+{k_{1}}^2\right)^2\Big), 
 \end{eqnarray*} 
 where $ Q(x_{12})$ is given as follows:
 \begin{eqnarray*} &&
 Q(x_{12})= \big(k^8 (p-3)^6+2 \left(7 k^2+2\right) k^6 (p-3)^5+\left(83 k^2+46\right)
   k^6 (p-3)^4\\ &&
   +8 \left(33 k^4+28 k^2-3\right) k^4 (p-3)^3
   +8 \left(59 k^6+71 k^4-21 k^2-1\right) k^2 (p-3)^2
   \\ &&
   +32 \left(14 k^6+23 k^4-12 k^2-1\right) k^2 (p-3)+16 \left(11
   k^8+24 k^6-18 k^4-1\right)\big){x_{12}}^{12}
    \\ &&
  +  \big(3 k^6 (p-3)^4+\left(29 k^6+10 k^4\right)
   (p-3)^3+\left(108 k^6+72 k^4\right) (p-3)^2
   \\ &&
   +\left(180 k^6+192 k^4-36
   k^2\right) (p-3)+112 k^6+184 k^4-96 k^2-8
    \big){x_{12}}^{10}
    \\&&
  + \big( 15 k^6 (p-3)^4+\left(140 k^6+40 k^4\right)
   (p-3)^3+\left(498 k^6+276 k^4\right) (p-3)^2 
    \\ &&
     +\left(792 k^6+672
   k^4-72 k^2\right) (p-3)+472 k^6+568 k^4-168 k^2-8\big){x_{12}}^{8}
   \\&&
 +  \big(5 k^4 (p-3)^3+\left(35 k^4+10 k^2\right) (p-3)^2+\left(83
   k^4+46 k^2\right) (p-3)+66 k^4+56 k^2-6\big){x_{12}}^{6}
      \\&&
  +  \big(15 k^2 (p-3)^2+\left(70 k^2+20\right) (p-3)+83 k^2+46
    \big){x_{12}}^{4}
       \\&&
   + \big(3 k^2 (p-3)+7 k^2+2
    \big){x_{12}}^{2} +k^2 {k_1}^6  > 0. 
  \end{eqnarray*}
 \end{proof}
 From the above it follows that the  non naturally reductive Einstein metrics obtained in Section 4 are not isometric.

 \bigskip
 \noindent
{\bf Acknowledgment.} The second author was supported by JSPS KAKENHI Grant Number JP21K03224.

 \bigskip
 \noindent
{\bf Statements and Declarations}

 \medskip
 \noindent
{\bf Conflict of interest}  The authors have no relevant financial or non-financial interests to disclosure.  

 \medskip
 \noindent
{\bf Declaration of generative AI and AI-assisted technologies in the writing process} During the preparation of this work the authors did not use any
AI technologies.

  

\end{document}